\documentclass[11pt,reqno]{article}

\usepackage{a4wide}
\usepackage{amsmath}
\usepackage{amssymb}
\usepackage{amsthm}
\usepackage{enumerate}
\usepackage{graphicx}
\usepackage{color}

\def\longrightharpoonup{\relbar\joinrel\rightharpoonup}

\DeclareMathAlphabet{\mathboit}{OT1}{cmr}{bx}{it}

\numberwithin{equation}{section}
\numberwithin{table}{section}

\newtheorem{Lemma}[equation]{Lemma}

\newtheorem{Theorem}[equation]{Theorem}

\newtheorem{Proposition}[equation]{Proposition}
\theoremstyle{definition}
\begingroup
\newtheorem{Remark}[equation]{Remark}
\newtheorem{Definition}[equation]{Definition}
\endgroup

\title{\textbf{Hexagonal patterns in a simplified model for block copolymers}}

\author{D.~P.~Bourne$^1$, M.~A.~Peletier$^2$, S.~M.~Roper$^1$}

\date{\today}

\begin{document}
\maketitle

\begin{abstract}
In this paper we study a new model for patterns in two dimensions, inspired by diblock copolymer melts with a dominant phase. The model is simple enough to be amenable not only to numerics but also to analysis, yet sophisticated enough to reproduce hexagonally packed structures that resemble the cylinder patterns observed in block copolymer experiments.

 Starting from a sharp-interface continuum model, a nonlocal energy functional involving a Wasserstein cost, we derive the new model using Gamma-convergence in a limit where the volume fraction of one phase tends to zero. The limit energy is defined on atomic measures; in three dimensions the atoms represent small spherical blobs of the minority phase, in two dimensions they represent thin cylinders of the minority phase.

 We then study minimisers of the limit energy. Numerical minimisation is performed in two dimensions by recasting the problem as a computational geometry problem involving power diagrams. The numerical results suggest that the small particles of the minority phase tend to arrange themselves on a triangular lattice as the number of particles goes to infinity. This is proved in the companion paper \cite{BournePeletierTheil} and agrees with patterns observed in block copolymer experiments.
  This is a rare example of a nonlocal energy-driven pattern formation problem in two dimensions where it can be proved that the optimal pattern is periodic.

  \end{abstract}

%
%
%----------------------------------------------------------------------------------------------------------------------------
%
%

\section{Introduction}
\label{Intro}

\footnotetext[1]{School of Mathematics and Statistics,
University of Glasgow,
15 Univ.~Gardens,
Glasgow G12 8QW, UK.}

\footnotetext[2]{Department of Mathematics and Computer Science and Institute for Complex Molecular Systems, Technische Universiteit Eindhoven,
PO Box 513, 5600 MB Eindhoven,
The Netherlands.}

Block copolymers are a famous example of a pattern-forming system and even appear in the popular-science literature \cite{Ball}.
A diblock copolymer molecule consists of a polymer chain of type A bonded covalently to another polymer chain of type B. Because of a repulsive force between the A and B chains, in diblock copolymer mixtures the A and B parts separate at the microscale to produce a wide variety of patterns including lamellae, cylinders, gyroids, and spheres.
See Figure~\ref{fig:CPW}.
\begin{figure}[h]
\def\fh{2.5cm}
\centering
\includegraphics[height=\fh]{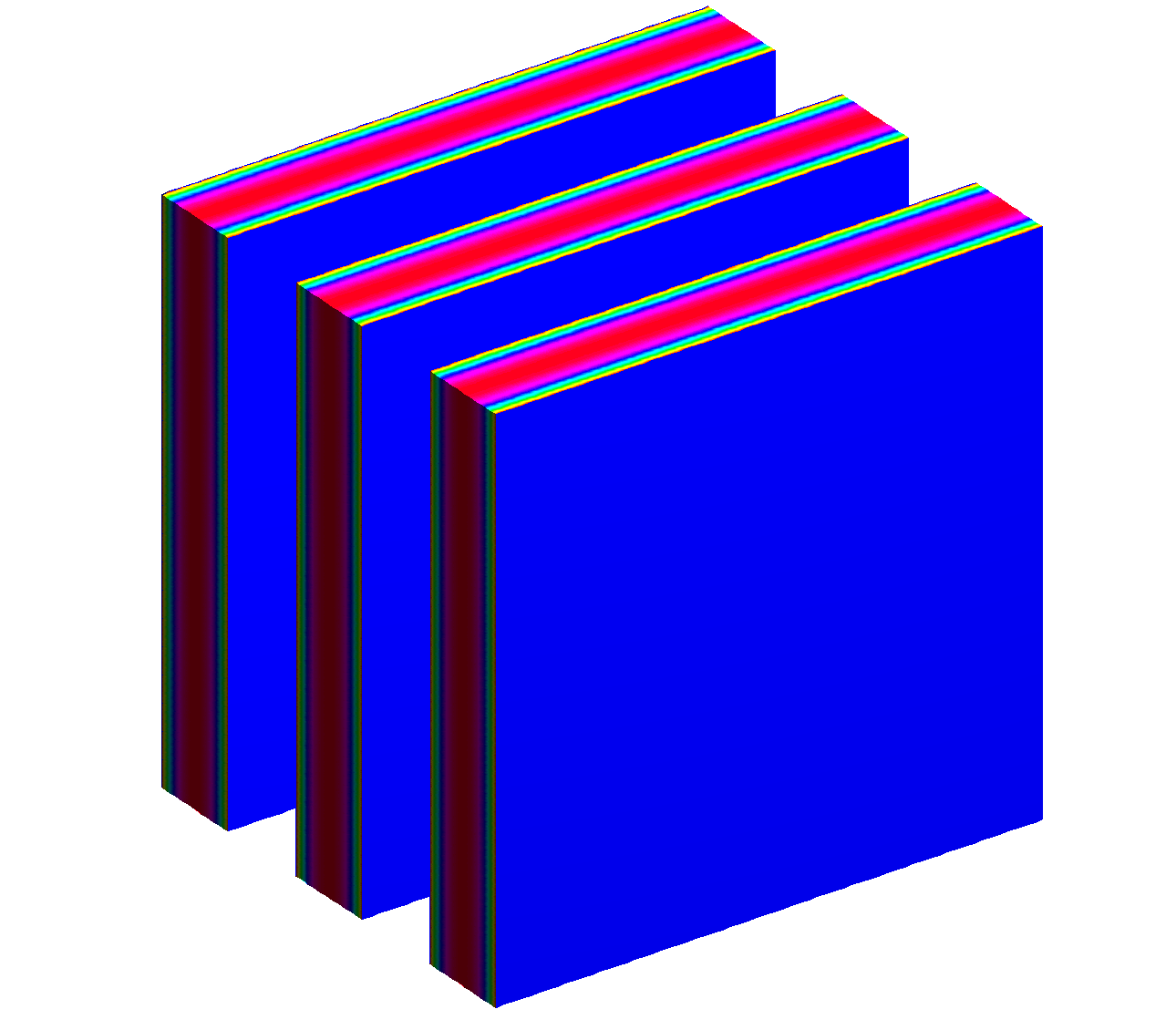}
\includegraphics[height=\fh]{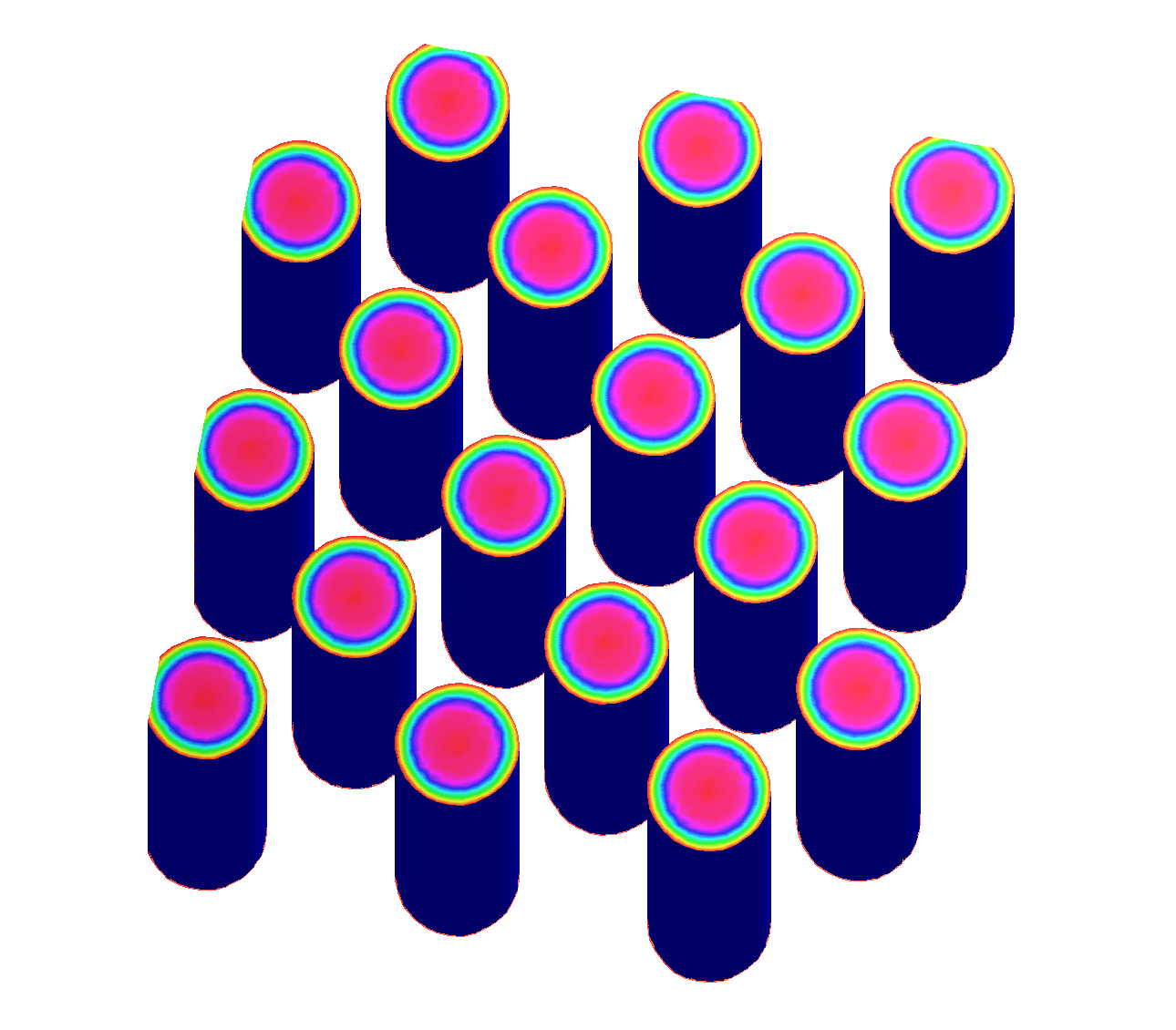}
\includegraphics[height=\fh]{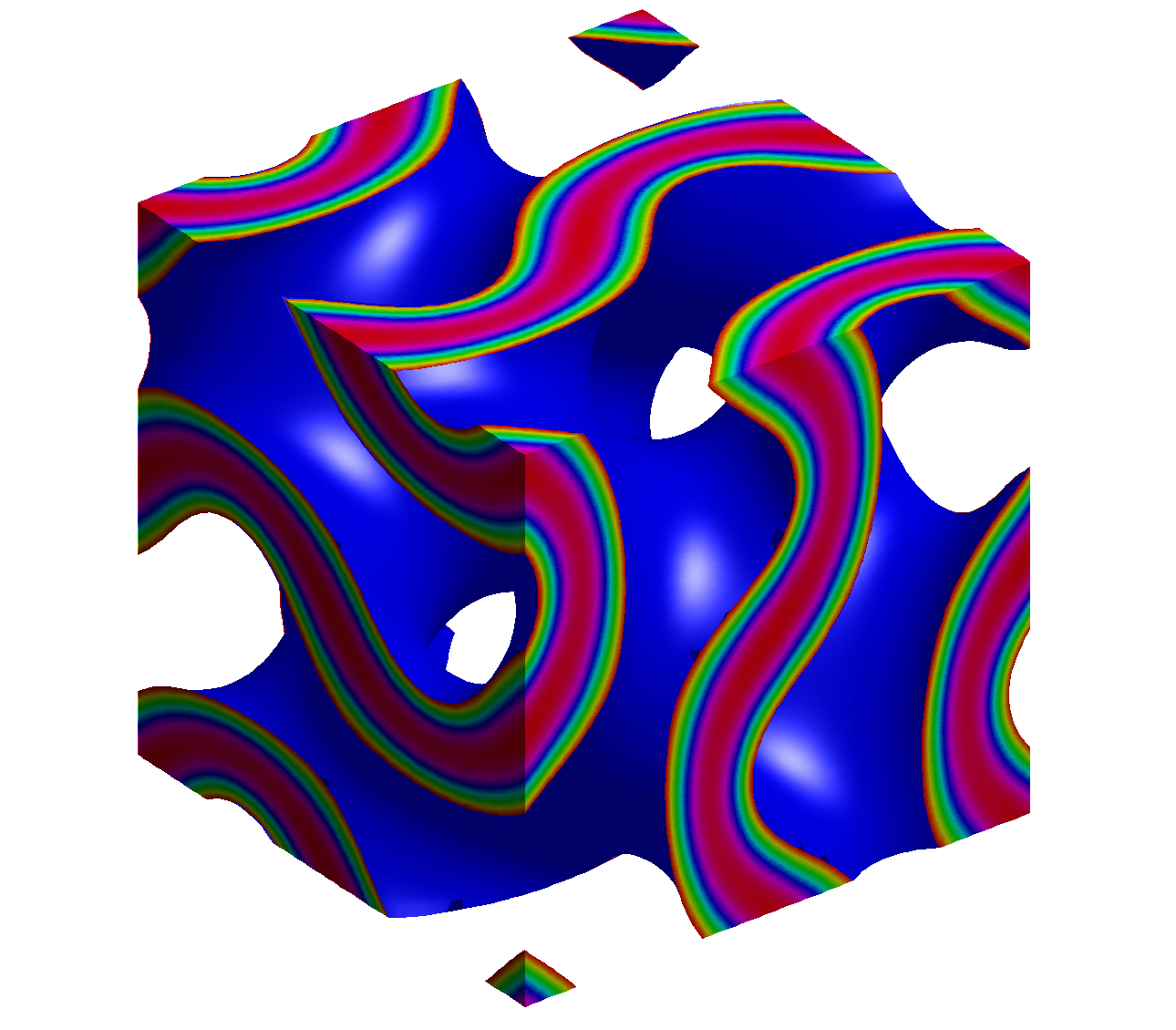}
\includegraphics[height=\fh]{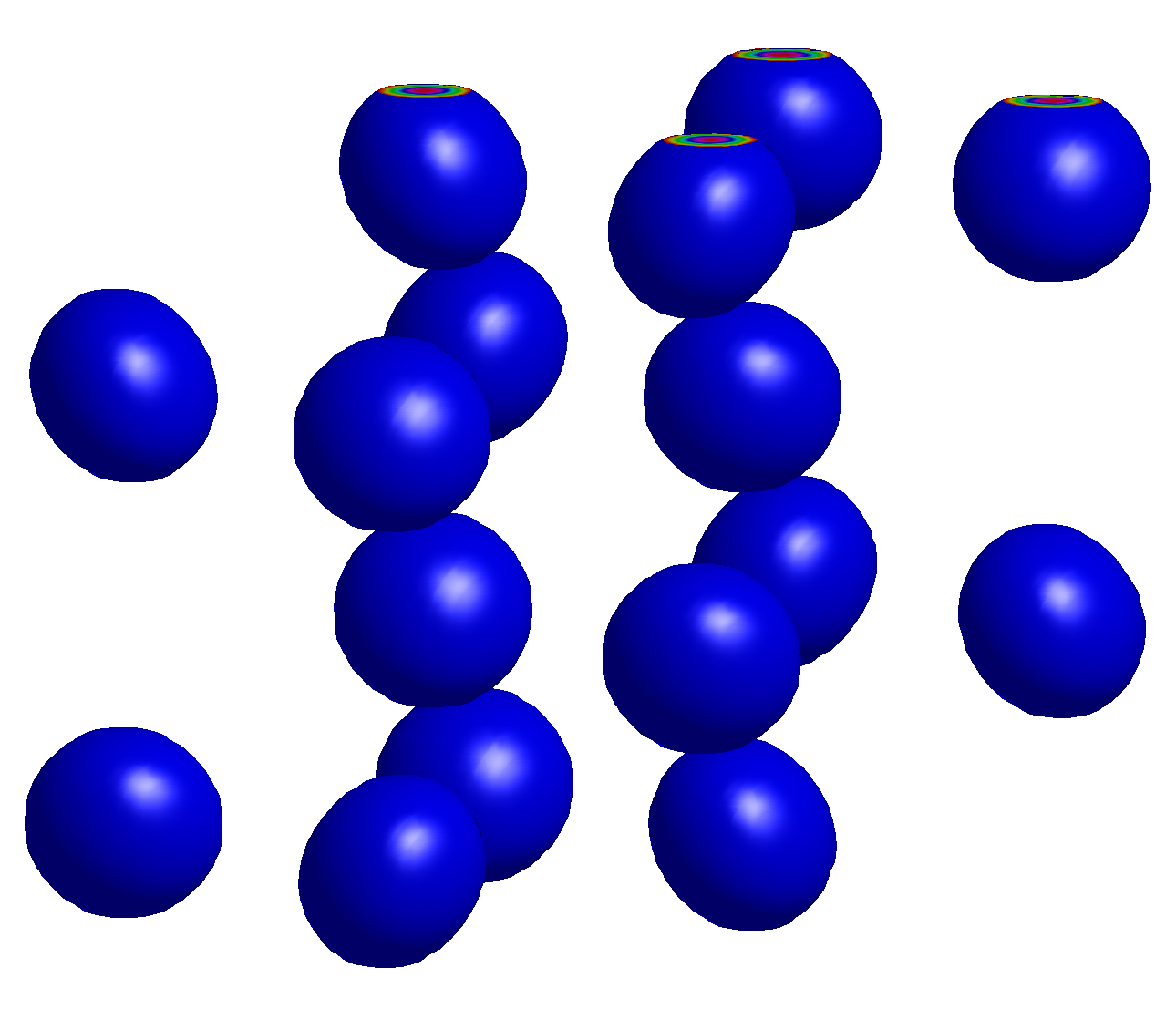}
\caption{Some examples of three-dimensional patterns arising in models of block copolymer melts. These figures have been obtained in~\cite{ChoksiPeletierWilliams09} by numerical minimisation of the Ohta-Kawasaki energy functional, a functional that is closely related to the models of this paper.}\label{fig:CPW}
\end{figure}

This spontaneous pattern formation occurs because the A and B blocks repel each other and so try to arrange themselves to be as far away from each other as possible under the constraint that each A block is bonded covalently to a B block. Apart from being beautiful examples of a pattern forming system, block copolymers
have the potential to be used as microscale structuring agents to develop new materials with prescribed macroscale properties~\cite{BatesFredrickson}.

In this paper we derive a new, idealised model for diblock copolymers. The model is applicable in the parameter regime where one phase has a large volume fraction and where the repulsive force between the A and the B blocks is large. The starting point in the derivation is the following sharp-interface continuum model.

\paragraph{The starting point: a sharp-interface continuum model.}
Let $\Omega$ be the domain occupied by the diblock copolymer melt, which we take to be an open, connected and bounded subset of $\mathbb{R}^d$, $d \ge 2$, with Lipschitz boundary.
 Let $u \in BV(\Omega;\{0,1\})$ be a phase indicator, where the
support of $u$ represents the region occupied by the A blocks, and the support of $1-u$
represents the region occupied by the B blocks. Let $m \in (0,1)$ be the mass fraction of A:
\[
m = \frac{1}{|\Omega|}\int_{\Omega} u(x) \, dx.
\]
This represents the average volume of the A blocks compared to the B blocks.
We assume that  $u = 0 \textrm{ on } \partial \Omega$, implying that the boundary prefers to be in contact with the B phase.
Let $p \in [1,\infty)$.
We assign the following energy to $u$:
\begin{equation}
\label{M.1}
E_{p,m}(u) =  m^{\frac{1-d}{d}} \int_{\Omega} | \nabla u | + W_p \left( \frac{1-u}{1-m} , \frac{u}{m} \right)
\end{equation}
where $W_p$ is the $p$-Wasserstein transport cost, which is defined as follows:
\[
W_p(\mu,\nu) = \inf \left\{ \int_\Omega |x - T(x)|^p \, d\mu \; : \; T:\Omega \to \Omega, \; T_\# \mu  = \nu \right\}
\]
for measures $\mu$, $\nu$ such that $\mu(\Omega) = \nu(\Omega)= |\Omega|$ and $\mu$ is absolutely continuous with respect to the Lebesgue measure.
Note that $W_p$ is the $p$-th power of the $p$-Wasserstein metric (e.g. \cite{Villani03}).

The energy \eqref{M.1} was derived in \cite{PeletierVeneroni10}. In this model a polymer is represented as two spheres connected by a bond, and a free energy is postulated that takes into account entropy, the bond energy, repulsion between the A and B spheres, and volume exclusion. The bond is modeled as a spring with energy $|e|^p$ in terms of the spring extension $e$; $p=2$ therefore corresponds to a linear spring. The A-B repulsion and the volume exclusion are represented by interaction with an external field that is self-consistently generated by the spheres themselves, i.e.\ in a mean-field manner. After taking a strong-segregation limit one finds~\eqref{M.1}.

The two terms in $E_{p,m}$ in~\eqref{M.1} can both be traced directly back to the modelling ingredients. The first term arises from the tradeoff between entropy and A-B repulsion, and prefers phase separation (since it penalises the perimeter of the support of $u$). The second term is the total energy in the A-B bonds and prefers phase mixing. The competition between the two terms determines the pattern.

\medskip

This model is strongly related to the well-studied Ohta-Kawasaki model~\cite{OhtaKawasaki86,BatesFredrickson90,FredricksonBates96,ChoksiRen03}. Although the Ohta-Kawasaki model is derived and often studied for the regime of diffuse A-B interfaces, much of the analytical work is on the simpler sharp-interface version, and many of the sharp-interface results carry over to the diffuse-interface case (see e.g.~\cite{ChoksiPeletier10,ChoksiPeletier11,Muratov10} for examples of this).

The main distinction between the Ohta-Kawasaki model and~\eqref{M.1} therefore lies in the form of the non-local term. The Wasserstein-$p$ distances $W_p$ are conceptually similar to the $W^{-1,2}$ Sobolev norm in the Ohta-Kawasaki model. One way of making this similarity explicit is by observing that when only considering bounded functions on bounded sets, convergence in the $W_p$ Wasserstein distances and in the Sobolev $W^{-1,s}$ norms is equivalent to each other and to weak  convergence in $L^q$.

A second relation is
shown in Figure~\ref{fig:SobolevWasserstein}. The $1$-Wasserstein distance is also the norm in $W^{-1,1}$, which we define here as the Sobolev norm dual to the $W^{1,\infty}$ seminorm: $W_1(\mu_1,\mu_2) = \sup\{ \int \varphi\, d(\mu_1-\mu_2): \|\nabla \varphi\|_\infty \leq 1\} =: \|\mu_1-\mu_2\|_{W^{-1,1}}$ (see e.g.~\cite[Section~1.2]{Villani03}). For increasing $p$, the Wasserstein-$p$ distances increase in strength: $W_{p'}(\mu_1,\mu_2)\geq W_p(\mu_1,\mu_2)$ if $p'\geq p$. At the same time, the $W^{-1,s}$-norm becomes weaker as~$s$ increases.

\begin{figure}[htb]
\centering
\includegraphics[height=3cm]{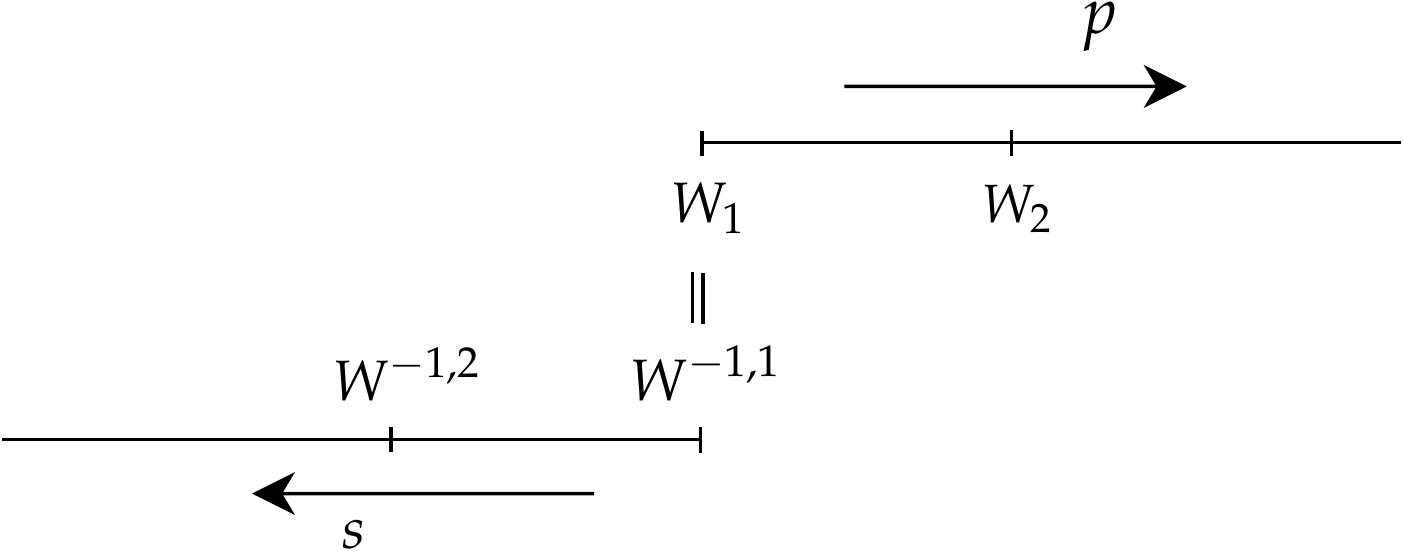}
\caption{A comparison of Wasserstein and Sobolev nonlocal terms; see the text for discussion.}
\label{fig:SobolevWasserstein}
\end{figure}
These arguments show how the Wasserstein nonlocal terms can be considered related to the Sobolev norms, and the model we study here is related to the Ohta-Kawasaki case of $W^{-1,2}$.

\paragraph{The vanishing-volume-fraction limit.}

We will consider the limit $m \to 0$, i.e., the limit in which phase A has vanishing volume fraction. The energy \eqref{M.1} is scaled in such a way that we see a finite number of particles of A in the limit. Similar limits have been studied for the Ohta-Kawasaki model~\cite{RenWei07,ChoksiPeletier10,ChoksiPeletier11,Muratov10,Niethammer,GlasnerChoksi,CicaleseSpadaro}.

It is convenient to rewrite the energy in the following way: Define $v := \frac um$, $n := \frac 1m$, and
$F_{p,n}(v) := E_{p,{1/n}}(v/n)$ so that
\begin{equation}
\label{M.3}
F_{p,n}(v) = \left\{
\begin{array}{ll}
\displaystyle n^{-\frac{1}{d}} \int_{\Omega} | \nabla v | + W_p \left( \frac{n-v}{n-1} , v \right)
& \textrm{if } v \in K_n, \\
+ \infty & \textrm{otherwise,}
\end{array}
\right.
\end{equation}
where
\begin{equation}
K_n := \left\{ v \in BV(\Omega;\{0,n\}): \frac{1}{|\Omega|}\int_{\Omega} v(x) \, dx = 1, \quad v = 0 \textrm{ on } \partial \Omega \right\}.
\end{equation}
This energy is the starting point for our analysis.

\paragraph{The new model.} By taking the $\Gamma$-limit of $F_{p,n}$ as $n \to \infty$ we obtain the following model. Recall that sending $n \to \infty$ corresponds to sending the volume fraction of phase A to zero, in which case
 phase A can be represented by a measure of the form
 \[
 \nu = \sum_{i=1}^M m_i \delta_{x_i} \quad \textrm{with} \quad  \sum_{i=1}^M m_i = | \Omega |.
 \]
 In three dimensions the points $x_i \in \Omega$ represent small spherical blobs of phase A in a sea of phase B, and the weights $m_i>0$ represent the relative size of the blobs. In two dimensions the points $x_i$ represent thin cylinders of phase A in a cross-section orthogonal to the cylinders. In
Theorem \ref{Gamma} we show that $F_{p,n}$ Gamma-converges to an energy of the form
\begin{equation}
\label{I.1}
F_p(\nu) = \lambda \sum_i m_i^\frac{d-1}{d} + W_p(1,\nu)
\end{equation}
where $\lambda > 0$ and the $1$ in $W_p(1,\nu)$ denotes the Lebesgue measure on $\Omega$.
Note that the number $M$ in the definition of $\nu$ is not prescribed and is an unknown of the problem.
The first term of $F_p$ represents the repulsion between the A and B blocks and is minimised when $M=1$, which corresponds to complete phase separation where there is just one particle of phase A. The second term of $F_p$ represents the covalent bonds between the A and B blocks and can be
made arbitrarily small by taking $M \to \infty$, corresponding to complete phase mixing, i.e., infinitely many particles of A equidistributed in B.
The competition between the two terms and the parameter $\lambda$ determines the nature of the minimisers.

\paragraph{Minimisers of the limit energy.}
From Section \ref{E} onwards we study minimisers of the limit energy $F_p$.

\medskip

\emph{Numerical Minimisation.}
Minimising $F_p$ numerically is challenging since (a) $F_p$ has infinitely many local minimisers,  and (b) it is difficult to evaluate numerically; evaluating the $p$-Wasserstein distance is equivalent to solving an infinite-dimensional linear programming problem. The second difficulty is addressed by using a deep a connection between the $p$-Wasserstein cost and generalised Voronoi diagrams (Proposition \ref{char}). This connection, made in \cite{Merigot11} for the case $p=2$, seems to be little known in the theoretical optimal transportation community. It allows the minimisation of $F_p$ to be reformulated as a computational geometry problem. By combining this formulation with Euler-Lagrange equations for $F_p$ (Theorem \ref{GlobalMin}), we derive an algorithm for finding stationary points of $F_p$ (Section \ref{CPD}). This algorithm is a generalisation of Lloyd's algorithm, which is a popular method for computing Centroidal Voronoi Tessellations (see Section \ref{lambda=0}).

For implementation purposes we limit our attention to the case $p=2$ in two dimensions.
Figure \ref{Minimisers} shows a number of minimisers $\nu$ of $F_2$ that were computed using our generalised Lloyd algorithm. The points represent the support of $\nu$ and the polygons represent the corresponding transport regions (the regions transported onto the support of $\nu$ by the optimal transport map $T$ for $W_2(1,\nu)$). We observe that, as $\lambda \to 0$, the support of $\nu$ tends to a triangular lattice and the transport regions tend to a hexagonal tiling. This agrees with the hexagonally packed cylinder patterns observed experimentally (see e.g.~\cite{BatesFredrickson}).

\begin{figure}[!h]
\begin{center}
\includegraphics[width=0.75\textwidth]{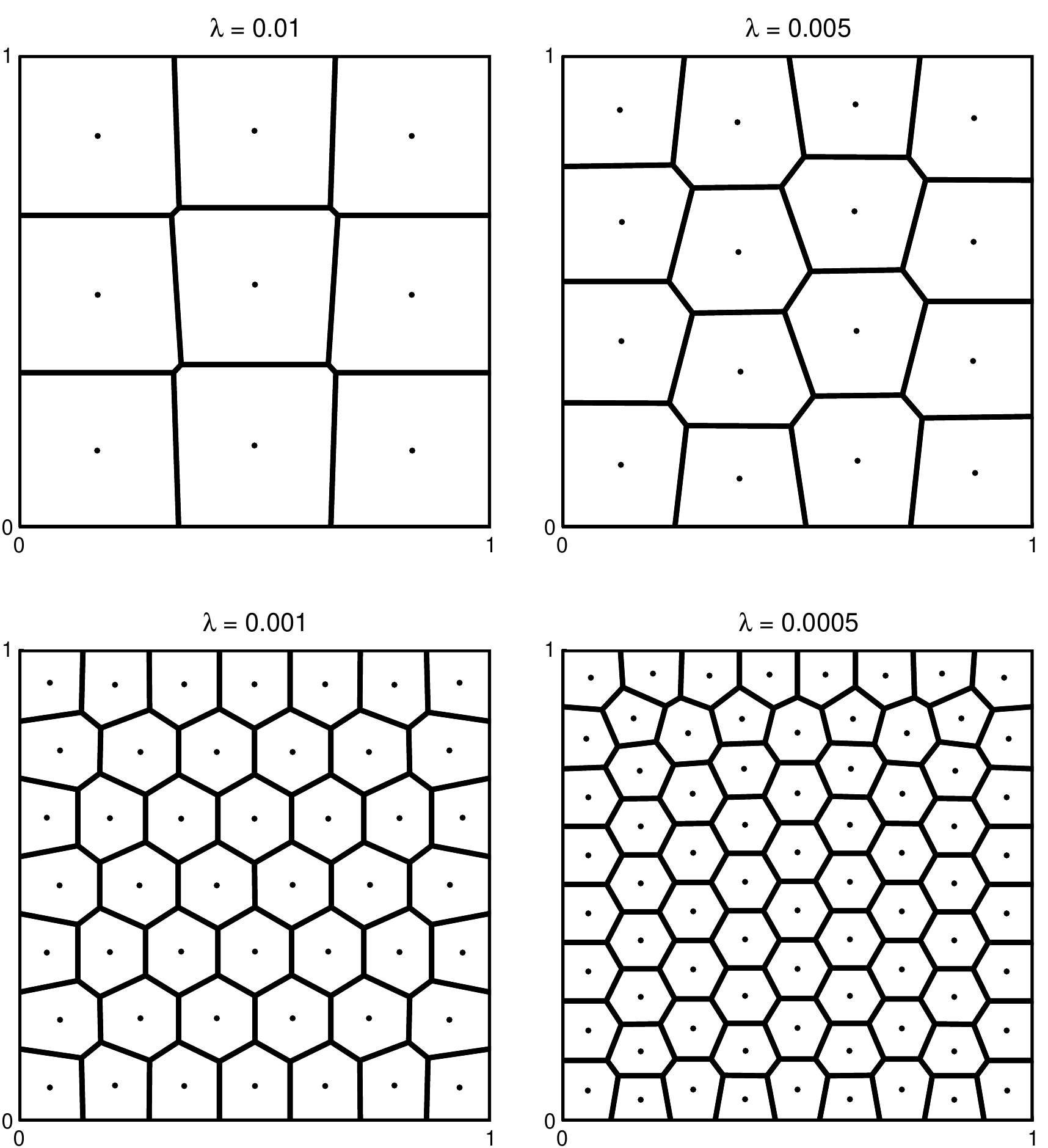}
\caption{\label{Minimisers} {Minimisers of $F_2$ (defined in equation \eqref{N.1}) in two dimensions.
We see that as $\lambda \to 0$ a hexagonal pattern appears. $F_2$ was minimised using a generalisation of Lloyd's algorithm, which is described in Section \ref{CPD} .}}
\end{center}
\end{figure}

\medskip

\emph{Rigorous characterisation of the pattern. } Figure~\ref{Minimisers}, and many other figures like it, strongly  suggest that the triangular lattice is optimal in some way, and that minimisers tend to approximate this optimal lattice, `as far as the boundary allows'.
This statement of \emph{crystallisation} we have formulated in an exact way, and proved with Florian Theil. A non-rigorous version is given in Section \ref{P}, and the full version and the proof are given in the companion paper \cite{BournePeletierTheil}.

\medskip

\paragraph{Related work.}
As we described above, the energy~\eqref{M.1} is related to the Ohta-Kawasaki energy, which is a nonlocal Cahn-Hilliard energy, in which the nonlocality is a negative Sobolev norm rather than a Wasserstein distance. This energy has been extensively studied, with a variety of rigorous and non-rigorous results. However, to our knowledge exact crystallisation results of the form \cite{BournePeletierTheil} do not exist; this appears to be the first time that it has been proved that multi-dimensional minimisers of an energy of  diblock copolymer type are periodic. Until now the closest results were the weak periodicity results of \cite{AlbertiChoksiOtto09} and \cite{Spadaro}. Note, however, that the diblock copolymer models analysed in these papers are different from and more sophisticated than the limit model we study here. It has recently been proved \cite{MoriniSternberg} that stripe patterns are optimal for the sharp interface limit of the Ohta-Kawasaki energy for thin two-dimensional domains.

In three dimensions we have no rigorous results, and we know of none either; however, we conjecture that minimisers of $F_2$ in three dimensions form small spheres centred on a body-centred cubic lattice (see Section \ref{3D}). This is also observed in experiments. Verifying this conjecture numerically will be the subject of a future paper.

This paper (along with the companion paper \cite{BournePeletierTheil}) gives an example of an energy-driven pattern formation problem in two dimensions where it can be proved that the optimal pattern is periodic. Some previous results for problems arising in materials science include \cite{Radin81}, \cite{Theil06} and \cite{YeungFrieseckeSchmidt12}. Proofs in three dimensions are even rarer  \cite{HalesHarrisonMcLaughlinNipkowObuaZumkeller10}.
Also, while the hexagonal pattern appears in a wide range of situations, again there are few rigorous results \cite{Fejes-Toth72,Gruber99,G-Fejes-Toth01,MorganBolton02}.

The energy \eqref{I.1} and the results in this paper have applications beyond diblock copolymers. Energies of the form of our limit energy \eqref{I.1} arise in optimal location problems \cite{LoschWoglomStolper54,ButtazzoSantambrogio,BouchitteJimenezMahadevan}, quantization and image processing \cite{Gersho79,Merigot11}
and are related to many other problems in computational geometry \cite{DuFaberGunzburger99}.

\paragraph{Outline of the paper.} In Section \ref{C} the discrete limit energy \eqref{I.1} is derived from the continuum energy \eqref{M.3} using Gamma-convergence. Euler-Lagrange equations for the limit energy are derived in Section \ref{E}.
In Section~\ref{N} we derive an algorithm for computing stationary points of the limit energy and implement it in two dimensions for the case $p=2$.
Section~\ref{P} reports results from the companion paper~\cite{BournePeletierTheil}, where minimisers are characterised analytically.
Finally, in Section \ref{3D} we make a conjecture about minimisers in three dimensions.
%
%
%
%
%
%----------------------------------------------------------------------------------------------------------------------------------------------
%
%
%
\section{The Small Volume Fraction Limit Model}
\label{C}
In this section we find the $\Gamma$-limit of $F_{p,n}$ as $n \to \infty$, i.e., as the volume fraction of phase~A tends to zero. By doing this we obtain a simpler model that is more amenable to numerics and analysis.
\begin{Theorem}[Compactness]
\label{Comp}
Let $\{ v_n \} \subset K_n$ be a sequence with bounded energy:
\begin{equation}
\nonumber
\sup_{n>0} F_{p,n}(v_n) \le C < \infty.
\end{equation}
Then $v_n dx \stackrel{*}{\longrightharpoonup} \nu$, where
\begin{equation}
\nonumber
\nu = \sum_i m_i \delta_{x_i}, \quad m_i \ge 0
\end{equation}
is an at most countable sum of dirac masses located at points $x_i \in \overline{\Omega}$. Moreover
\begin{equation}
\nonumber
\sum_{i} m_i = |\Omega|, \quad \quad \sum_{i} m_i^{\frac{d-1}{d}} < \infty.
\end{equation}
\end{Theorem}
\begin{proof}
This follows almost immediately from the Second Concentrated Compactness Lemma of P.~L.~Lions \cite{Lions85}. From the given sequence
 $v_n$, define the rescaled sequence $w_n := n^{-\frac1d} v_n$.
 Since $v_n$ has bounded energy, the sequence $\nabla w_n$ is bounded in $\mathcal{M}$, the space of bounded Radon measures. Also
 \begin{equation}
 \int_{\Omega} w_n \, dx = n^{-\frac1d} \int_{\Omega} v_n  \, dx = n^{-\frac1d} |\Omega| \to 0,
 \end{equation}
  which implies that $w_n \to 0$ in $L^1(\Omega)$.  Note that
 \begin{equation}
w_n^{1^*} = w_n^{\frac{d}{d-1}} = v_n.
\end{equation}
Since $\int_\Omega v_n \, dx = |\Omega|$, the sequence $w_n^{1^*} dx \equiv v_n dx$ is bounded and so $w_n^{1^*} dx \stackrel{*}{\longrightharpoonup} \nu$ for some $\nu \in \mathcal{M}$.
Therefore $w_n$ satisfies the assumptions of the Second Concentrated Compactness Lemma, which implies that the limit measure $\nu$ is of the form given in the assertion.
\end{proof}

\begin{Remark}[Simpler proof for the case $d=2$]
For the case $d=2$ there is an elementary proof that does not rely on the industrial-strength compactness lemma used above. In two dimensions the diameter of a connected set can be bounded by its perimeter, $2 \, \textrm{Diam}(S) \le \textrm{Per}(S)$. The energy bound on~$v_n$ implies that the perimeter of the support of $v_n$ tends to zero. Therefore the diameter of each connected component of the support of  $v_n$ tends to zero, which implies that the limit measure $\nu$ is a sum of dirac masses.
\end{Remark}

Let $K$ be the following set of Radon measures on $\mathbb{R}^d$:
\begin{equation}
\nonumber
K= \left\{ \nu = \sum_{i=1}^\infty m_i \delta_{x_i} \; : \; m_i \ge 0, \; x_i \in \overline{\Omega}, \;
x_i \ne x_j \textrm{ if } i \ne j, \;
\sum_{i=1}^\infty m_i = |\Omega|, \;
\sum_{i=1}^\infty m_i^{\frac{d-1}{d}} < \infty
\right\}.
\end{equation}
\begin{Theorem}[$\Gamma$-convergence: The small volume fraction limit] Let $\nu$ be a Radon measure on $\mathbb{R}^d$.
\label{Gamma}
Define
\begin{equation}
\nonumber
F_p(\nu) := \left\{
\begin{array}{ll}
d  \alpha(d)^{\frac1d} \sum_{i} m_i^{\frac{d-1}{d}}   + W_p \left( 1 , \nu \right)
& \textrm{if } \nu \in K, \\
+ \infty & \textrm{otherwise,}
\end{array}
\right.
\end{equation}
where $\alpha(d)$ is the volume of the unit ball in $\mathbb{R}^d$.
With respect to the topology of weak convergence of measures, $F_{p,n}$ Gamma-converges to $F_p$ as $n \to \infty$, i.e.,
\begin{enumerate}
\item[(i)] Let $v_n \in K_n$ satisfy $v_n dx \stackrel{*}{\longrightharpoonup} \nu$ for some $\nu \in K$. Then
\begin{equation}
\nonumber
F_p (\nu) \le \liminf_{n \to \infty} F_{p,n} (v_n).
\end{equation}
\item[(ii)] Given $\nu \in K$, there exists a sequence $v_n$ in $K_n$ such that $v_n dx \stackrel{*}{\longrightharpoonup} \nu$ and
    \begin{equation}
    \nonumber
    \lim_{n \to \infty} F_{p,n} (v_n) = F_p (\nu).
    \end{equation}
\end{enumerate}
\end{Theorem}
\begin{proof}
First we prove (ii). By approximation (see e.g.~\cite[Remark~1.29]{Braides02}) we can assume that $\nu$ is a finite sum of dirac masses: $\nu = \sum_{i=1}^M m_i \delta_{x_i} \in K$. We can also assume that none of the points $x_i$ belong to the boundary of $\Omega$.
Define $r_i$ by $r_i^d = \frac{m_i}{n \alpha(d)}$ so that the ball $B(x_i,r_i)$ has volume $m_i/n$.
Let $v_n$ be the function taking values in $\{ 0, n\}$ with support $\overline{\bigcup_i B(x_i,r_i)}$.
Note that, for $n$ sufficiently large, the balls $B(x_i,r_i)$ are disjoint and are contained in $\Omega$. Therefore $v_n \in K_n$. It is easy to check that
\begin{equation}
v_n dx \stackrel{*}{\longrightharpoonup} \nu, \quad \frac{n-v_n}{n-1} \stackrel{*}{\longrightharpoonup} 1.
\end{equation}
Since the $p$-Wasserstein distance metrizes weak convergence of measures (see \cite[Theorem 7.12]{Villani03}), it follows that
\begin{equation}
\label{C.7}
W_p \left( \frac{n-v_n}{n-1}, v_n \right)
\to W_p(1,\nu).
\end{equation}
Recall that $d \alpha(d)$ is the surface area of the unit ball in $\mathbb{R}^d$. Therefore for all $n$
\begin{equation}
\label{C.8}
n^{-\frac{1}{d}} \int_\Omega | \nabla v_n | = d  \alpha(d)^{\frac{1}{d}} \sum_i m_i^{\frac{d-1}{d}}.
\end{equation}
Combining \eqref{C.7} and \eqref{C.8} yields (ii).

Now we turn our attention to (i).
Let $v_n dx \stackrel{*}{\longrightharpoonup} \nu$, where $\nu = \sum_{i=1}^\infty m_i \delta_{x_i} \in K$.
Fix $M\in \mathbb{N}$. By Lemma 5.3 in \cite{ChoksiPeletier10} we can modify the sequence $v_n$ to obtain a sequence $\tilde{v}_n$ such that
\[
\tilde{v}_n = \sum_{i=1}^M v^n_i
\]
where $v^n_i \in BV(\Omega;\{ 0,n\})$ have disjoint supports,
$\textrm{dist}(\textrm{supp} \, v^n_i,\textrm{supp}\, v^n_j) > 0 \textrm{ for all } i \ne j$,
and satisfy
\begin{equation}
\label{wli}
\textrm{w-}\liminf_{n \to \infty} v^n_i \ge m_i \delta_{x_i}, \quad i \in \{1, \ldots , M \}
\end{equation}
in the sense of distributions, and
\begin{equation}
\label{LB}
\int_{\Omega} |\nabla v_n| \ge \int_{\Omega} |\nabla \tilde{v}_n |.
\end{equation}
This modification is necessary so that we can apply the Isoperimetric Inequality (with optimal constant) as follows:
\begin{equation}
\label{iso}
\begin{aligned}
n^{-\frac 1d}\int_{\Omega} |\nabla \tilde{v}_n | = n^{-\frac 1d} \sum_{i=1}^M \int_{\mathbb{R}^d} | \nabla v_i^n |
& \ge n^{-\frac 1d} \sum_{i=1}^M d \alpha(d)^{\frac 1d} \left( \int_{\mathbb{R}^d} |v_i^n|^{\frac{d}{d-1}} \, dx \right)^{\frac{d-1}{d}} \\
& = d \alpha(d)^{\frac 1d} \sum_{i=1}^M \left( \int_{\mathbb{R}^d} v_i^n \, dx \right)^{\frac{d-1}{d}},
\end{aligned}
\end{equation}
where the last equality holds since $v^n_i$ takes values is $\{0,n \}$. Equations \eqref{LB} and \eqref{iso} imply that
\[
F_{p,n}(v_n) \ge d \alpha(d)^{\frac 1d} \sum_{i=1}^M \left( \int_{\mathbb{R}^d} v_i^n \, dx \right)^{\frac{d-1}{d}} + \, W_p \left( \frac{n-v_n}{n-1},v_n \right).
\]
Therefore by using \eqref{wli} and the fact that the Wasserstein distance metrizes weak convergence of measures we obtain
\[
\liminf_{n \to \infty} F_{p,n}(v_n) \ge d \alpha(d)^{\frac 1d} \sum_{i=1}^M m_i^{\frac{d-1}{d}} + W_p(1,\nu).
\]
This holds for all $M \in \mathbb{N}$. Therefore it also holds for $M=\infty$ and we obtain the desired result.
\end{proof}

\begin{Remark}[Existence of global minimisers for $F_p$ and $F_{p,n}$]
\label{EGM}
It is clear that $F_p$ has a global minimiser since it is a $\Gamma$-limit and so is lower semicontinuous. It is also easy to see that $F_{p,n}$
has a global minimiser: Any infimising sequence $\{v_k\} \subset K_n$ is bounded in $BV(\Omega)$ and so has a strongly convergent subsequence $v_{k_j} \to v$ in $L^1(\Omega)$. The variation measure is lower semicontinuous with respect to this convergence, $|\nabla v|(\Omega) \le \liminf_{j \to \infty} |\nabla v_{k_j}|(\Omega)$ (see \cite[page 172, Theorem 1]{EvansGariepy}), and the Wasserstein cost $W_p$ is continuous with respect to this convergence, therefore $F_{p,n}(v)=\inf_{K_n} F_{p,n}$.
\end{Remark}

\begin{Remark}[Minimisers of $F_2$ are supported on a finite set]
It is shown in \cite[Lemma 7(i)]{BournePeletierTheil} that if $\nu$ is a global minimiser of $F_2$, then $\nu$ is supported on a finite set (as opposed to just a countable set). This is shown by proving a positive lower bound on $\inf_i m_i$.
\end{Remark}
%
%
%
%--------------------------------------------------------------------------------------------------------------------
%
%
\section{Euler-Lagrange Equations}
\label{E}
In the rest of the paper we study the limit energy $F_p$.
We start by deriving two Euler-Lagrange equations.

\begin{Proposition}[Euler-Lagrange equation obtained by varying $x_i$]
\label{Prop1}
Let $\nu = \sum_i m_i \delta_{x_i} \in K$ be a minimiser of $F_p$ such that $x_i \in \Omega$ for all $i$, i.e., $x_i \notin \partial \Omega$. Let $V_i \subseteq \Omega$ be the set of points transported to $x_i$ by the optimal transport map
$T$ for $W_p(1,\nu)$, i.e., $V_i = T^{-1}({x_i})$. Then for each $i$
\begin{equation}
\label{EL1}
0 = \int_{V_i} (x_i-x) |x_i-x|^{p-2} \, dx.
\end{equation}
\end{Proposition}
For the case $p=2$ this says that each mass $x_i$ is located at the centre of mass, or centroid, of its transport region $V_i$. The assumption that minimisers have no $x_i \in \partial \Omega$ is necessary. For example, if $\Omega$ is an annulus with $|\Omega|$ sufficiently small, then global minimisers $\nu$ of $F_p$ are supported at just one point, which lies on the inner boundary of the annulus (there are infinitely many global minimisers). The centre of mass of the annulus, however, lies at its centre, and so \eqref{EL1} is not satisfied.

\begin{proof}
We vary the positions of the dirac masses, but not their weights, using an inner variation along the same lines as \cite{JordanKinderlehrerOtto}.
Suppose that $\nu = \sum_i m_i \delta_{x_i}$ is a minimiser of~$F_p$. We consider variations of the form
$\nu_\tau = \Phi_{\tau \#} \nu$, where $\{ \Phi_\tau \}_{\tau \ge 0}$ is a 1-parameter family of smooth invertible maps from $\overline{\Omega} \to \overline{\Omega}$ such that $\Phi_0(x)=x$. To be precise, we define $\Phi_\tau$ through the following ODE. Let $\xi \in C_0^\infty(\overline{\Omega};\overline{\Omega})$. Define $\Phi_\tau(y) \equiv \Phi(\tau,y)$ by
\begin{equation}
\label{E.1}
\begin{aligned}
\frac{d}{d\tau} \Phi(\tau,y) & = \xi(\Phi(\tau,y)) \quad \tau \ge 0,\\
\Phi(0,y) & = y.
\end{aligned}
\end{equation}
Since in our case the measure $\nu$ has such a simple form,  the push-forward $\nu_\tau = \Phi_{\tau\#} \nu$ reduces to $\nu_\tau = \sum_i  m_i \delta_{\Phi_\tau (x_i)}$. Formally,
the Euler-Lagrange equation is
\begin{equation}
\label{E.2}
\frac{d}{d\tau}  F_p(\nu_\tau) |_{\tau = 0} = 0.
\end{equation}
Since we are not varying the weights $m_i$ of the dirac masses, the left-hand side of this equation reduces to $\partial_\tau W_p(1,\nu_{\tau}) |_{\tau=0}$. Computing this derivative rigorously requires some care. The same calculation as in
 \cite[p.~11]{JordanKinderlehrerOtto} shows that equation \eqref{E.2} leads to the following Euler-Lagrange equation:
\begin{equation}
\label{E.7}
0 = \int_{\Omega \times \Omega} \xi(y) \cdot (y-x) |y-x|^{p-2} \, d \gamma \quad \textrm{for all } \xi \in C_0^\infty
\end{equation}
where $\gamma$ is
an optimal transport plan for the pair $(1,\nu)$, i.e.,
\begin{equation}
\label{E.3}
W_p(1,\nu) = \int_{\Omega \times \Omega} |y-x|^p \, d \gamma(x,y).
\end{equation}

Now we rewrite equation \eqref{E.7} in the form of the proposition statement. If $T$ is an optimal transport map, so that $\nu = T_\# 1$ and $\gamma = (\mathrm{id} \times T)_\# 1$, then equation \eqref{E.7} can be written as
\begin{equation}
\label{E.8}
0 = \int_{\Omega} \xi(T(x)) \cdot (T(x)-x) |T(x)-x|^{p-2} \, dx \quad \textrm{for all } \xi \in C_0^\infty.
\end{equation}
Note that $T$ maps $\Omega$ onto spt$(\nu) = \bigcup_i \{x_i\}$. Let $V_i$ denote the set of points mapped to $x_i$ by $T$, i.e., $V_i=T^{-1}({x_i})$. Then equation \eqref{E.8} reduces to
\begin{equation}
\label{E.9}
0 = \sum_i \int_{V_i} \xi(x_i) \cdot (x_i-x) |x_i-x|^{p-2} \, dx \quad \textrm{for all } \xi \in C_0^\infty.
\end{equation}
Since this holds for all $\xi \in C_0^\infty$ and since $x_i \notin \partial \Omega$ for any $i$, we arrive at the Euler-Lagrange equations
\begin{equation}
\label{E.10}
0 = \int_{V_i} (x_i-x) |x_i-x|^{p-2} \, dx \quad \textrm{for all } i.
\end{equation}
\end{proof}

\begin{Proposition}[Euler-Lagrange equation obtained by varying $m_i$]
\label{Prop2}
Let $\nu = \sum_{i=1}^M m_i \delta_{x_i} \in K$ be a minimiser of $F_p$ consisting of a finite number of dirac masses.
Let $(\phi_*,\psi_*)$ be an optimal Kantorovich potential pair for the problem of transporting $1$ to $\nu$, i.e.,
\[
W_p(1,\nu) = \int_{\Omega} \phi_* (x) dx + \sum_i m_i \psi_*(x_i)
\]
and $\phi_*(x) + \psi_*(x_i) \le |x-x_i|^p$ for almost all $x \in \Omega$ and all $i$.
Then
\begin{equation}
\label{EL2}
 \alpha (d)^{\frac{1}{d}} (d-1) m_i^{-\frac{1}{d}}
+  \psi_*(x_i) = \mathrm{constant}
\end{equation}
where the constant is independent of $i$.
\end{Proposition}
The constant appearing in \eqref{EL2} is the Lagrange multiplier for the constraint $\sum_i m_i = | \Omega |$.
Its presence also agrees with the fact that $\phi_*$ and $\psi_*$ are only defined up to a constant: If $(\phi_*,\psi_*)$ is an optimal Kantorovich potential pair then so is $(\phi_*+c,\psi_*-c)$ for all $c \in \mathbb{R}.$

\begin{proof}
This time we keep the positions of the dirac masses fixed and vary their weights in the following way. Define
\begin{equation}
\label{E.11}
m^\eta_i = m_i + \eta n_i, \quad n_i \in \mathbb{R}, \quad \sum_i n_i = 0.
\end{equation}
Given $n_i$, for $\eta$ sufficiently small we have $m^\eta_i > 0$ for all $i$ and so $\nu^\eta := \sum_i m_i^\eta \delta_{x_i} \in K$.
Let $(\phi^\eta_*,\psi^\eta_*)$ be an optimal Kantorovich potential pair for the problem of transporting $1$ to $\nu^\eta$, i.e.,
\begin{equation}
\label{E.12}
W_p(1,\nu^\eta) = \int_\Omega \phi^\eta_* \, dx + \sum_i m_i^\eta \psi^\eta_* (x_i)
\end{equation}
and $\phi^\eta_*(x)+\psi^\eta_*(x_i) \le |x-x_i|^p$ for almost all $x \in \Omega$ and all $i$.
By adding and subtracting a constant to $\psi_*^\eta$ and $\phi_*^\eta$ we can assume that $\psi_*^\eta(x_1) = 0$, and similarly $\psi_*(x_1)=0$; by Lemma~\ref{lemma:convergence-Kant-pot} below, then $\psi_*^\eta(x_i)\to\psi(x_i)$ for all $i$.

Since $(\phi_*^\eta,\psi_*^\eta)$ is admissible, we can estimate
\begin{equation}
\label{E.14}
W_p(1,\nu) %= \sup_{\phi \in \mathrm{Lip}_1} \int_\Omega \phi \, d(1-\nu^\eta)
\ge
\int_\Omega \phi^\eta_* \, dx + \sum_i m_i \psi^\eta_*(x_i).
\end{equation}
Therefore by equations \eqref{E.11}--\eqref{E.14},
\begin{equation}
W_p(1,\nu^\eta) - W_p(1,\nu) \le \eta \sum_i n_i \psi^\eta_*(x_i).
\end{equation}
Therefore, since $\nu$ is a minimiser of $F_p$,
\begin{equation}
0 \le \frac{1}{\eta} [ F_p(\nu^\eta) - F_p(\nu) ]
\le \sum_{i} \left[ d \alpha(d)^{\frac 1d}  \frac{(m_i^\eta)^{\frac{d-1}{d}} - (m_i)^{\frac{d-1}{d}}}{\eta} + n_i \psi^\eta_*(x_i) \right].
\end{equation}
Taking the limit of the right-hand side as $\eta \to 0$ gives
\begin{equation}
\label{E.15}
0 \le  \sum_i \left[ \alpha (d)^{\frac{1}{d}} (d-1)m_i^{-\frac{1}{d}}
+ \psi_*(x_i) \right] n_i
\end{equation}
for all $n_i$ satisfying \eqref{E.11}.
Therefore
\begin{equation}
 \alpha (d)^{\frac{1}{d}} (d-1) m_i^{-\frac{1}{d}}
+  \psi_*(x_i) = \mathrm{constant}
\end{equation}
as required.
\end{proof}

\begin{Lemma}
\label{lemma:convergence-Kant-pot}
Let $\nu = \sum_{i=1}^M m_i\delta_{x_i}$ and $\nu^\eta = \sum_{i=1}^M m_i^\eta\delta_{x_i}$ satisfy $m_i^\eta\to m_i$ as $\eta\to0$. Let $(\phi_*,\psi_*)$ and $(\phi_*^\eta,\psi_*^\eta)$ be corresponding Kantorovich potentials, and assume that $\psi_*(x_1) = \psi_*^\eta(x_1) = 0$. Then $\psi_*^\eta(x_i)\to\psi_*(x_i)$ for all $i$.
\end{Lemma}

\begin{proof}
This result is a small extension of~\cite[Lemma~3.4]{ButtazzoSantambrogio}, who proved the result for  $p>1$. Here we extend it to all $p\geq1$.
The only new requirement is a proof that optimal potentials are unique for $p\geq1$, up to addition of constants, and we now show this. Fix $\nu=\sum_{i=1}^M m_i\delta_{x_i}$; note that the Kantorovich pair $(\phi_*,\psi_*)$ satisfies
\[
\phi_*(x) = \min_i \bigl(|x-x_i|^p - \psi_*(x_i)\bigr),\qquad\text{for }x\in\Omega,
\]
so that
\begin{align*}
W_p(1,\nu) &= \sup_{\{\psi(x_i)\}_i} \int_\Omega \min_i \bigl(|x-x_i|^p - \psi(x_i)\bigr)\, dx + \sum_i m_i \psi(x_i)\\
&= \sup_{a = \{a_i\}_{i=1}^M }\ \underbrace{\inf_{\{\Omega_i\}\text{ partition of }\Omega}\
\sum_i \int_{\Omega_i} \bigl(|x-x_i|^p - a_i\bigr)\, dx + \sum_i m_i a_i}_{f(a)}.
\end{align*}
The function $f$ is affine along lines of the form $a+b(1,\dots,1)$, for $b\in\mathbb R$, and strictly concave in all other directions.
This follows from remarking that if $\{\Omega_i\}_i$ is an optimal partition of $\Omega$ for $a = \{a_i\}_i$, then for a perturbed $a + \tilde a$ we have
\begin{align*}
& f(a+\tilde a) - f(a) =\\
&=\inf _{\{\widetilde \Omega_i\}\text{ partition of }\Omega}
\sum_i \int_{\widetilde \Omega_i} \bigl(|x-x_i|^p - a_i-\tilde a_i\bigr)\, dx
- \sum_i \int_{\Omega_i} \bigl(|x-x_i|^p - a_i\bigr)\, dx + \sum_i m_i \tilde a_i\\
&\leq \qquad \sum_i \int_{\Omega_i} \bigl(|x-x_i|^p - a_i-\tilde a_i\bigr)\, dx
- \sum_i \int_{\Omega_i} \bigl(|x-x_i|^p - a_i\bigr)\, dx + \sum_i m_i \tilde a_i\\
&= \qquad \sum_i \tilde a_i(m_i-|\Omega_i|),
\end{align*}
and the  inequality is strict whenever $\{\Omega_i\}_i$ is not an optimal partition for $a+\tilde a$, which is whenever the $\tilde a_i$ are not all equal. This strict inequality implies the strict concaveness of $f$ in non-constant directions, and therefore the $a_i$ (and the $\psi_*(x_i)$) are uniquely determined up to constants.
\end{proof}

\begin{Remark}[Stationary points also satisfy the Euler-Lagrange equations]
It can also be shown that stationary points of $F_p$ satisfy the Euler-Lagrange equation \eqref{EL2}, not only global minimisers.
This is shown by using that $\phi^\eta_* \to \phi_*$ uniformly, which follows from the fact that $\phi^\eta_*$ is $|x-y|^p$-concave (c.f.~\cite[Lemma 3.4]{ButtazzoSantambrogio}).
Then it follows that $F_p(\nu^\eta)$ is differentiable with respect to $\eta$ and
 so we do not need to use the fact that $\nu$ is a global minimiser, just that it is a stationary point. The Euler-Lagrange equation \eqref{EL1} also holds for stationary points.
\end{Remark}

%
%
%
%
%-----------------------------------------------------------------------------------------------------------------------------------------------
%
%
%
%

\section{Numerical Optimisation}
\label{N}
In the remainder of the paper we study minimisers of the limit energy $F_p$ derived in Theorem~\ref{Gamma}.
We characterise global minima and derive an algorithm to compute them for all $p \in [1,\infty)$, $d \ge 2$. For implementation purposes we limit
our attention to two dimensions, $d=2$, and to $p=2$. The algorithm will be implemented in three dimensions and for other values of $p$ in a future paper. The simple case $d=1$ is discussed in Remark \ref{Remark d=1}.

By rescaling so that $\Omega \subset \mathbb{R}^d$ has area 1, the energy $F_p$ has the form
\begin{equation}
\label{N.1}
F_p(\nu) = \lambda \sum_{i=1}^M m_i^{\frac{d-1}{d}} + W_p(1,\nu)
\end{equation}
for some constant $\lambda > 0$, where $\nu = \sum_{i=1}^M m_i \delta_{x_i}$ with $m_i >0$, $\sum_{i=1}^M m_i = 1$, $x_i \in \overline{\Omega}$, $x_i \ne x_j$ if $i \ne j$.
The parameter $\lambda$, which comes from the rescaling, can also be thought of as modelling the repulsive strength between the A and B blocks.
This parameter was suppressed in the original energy \eqref{M.3}.

Note that $M$ is not fixed
\emph{a priori}.
The first term of $F_p$ is minimised when $M=1$, i.e., when $\nu$ consists of just one dirac mass (placed anywhere). The minimum value of the second term of $F_p$ converges to $0$ as $M \to \infty$ (since the Lebesgue measure can be approximated arbitrarily well by dirac masses). The parameter $\lambda$ and the competition between the two terms determines the value of $M$ for minimisers and the minimising patterns. For example,
when $\lambda$ is large the first term of $F_p$ dominates and the minimiser is $\nu = \delta_{x_1}$, where $x_1$
 satisfies equation \eqref{EL1}. In particular, $M=1$. As $\lambda$ decreases, $M$ increases. The scaling of the energy suggests that $M \sim \lambda^{-\frac{d}{p+1}}$.

%
%
%-------------------------------------------------------------------------------------------------------------------------------------
%
%

\subsection{Minimisers when $\lambda=0$ and $M$ is fixed: Centroidal Voronoi Tessellations.}
\label{lambda=0}
To get an intuition for the problem we first consider the simpler problem where $\lambda=0$ and $M$ is prescribed. We also take $p=2$ to start with.
We assume that $\Omega$ is convex. This ensures that the Euler-Lagrange equations have a solution (see below).
In this case the energy reduces to
\[
F(\nu) = W_2(1,\nu)
\]
where $\nu = \sum_{i=1}^M m_i \delta_{x_i}$ for some $M \in \mathbb{N}$ fixed. (Note that $M$ must be fixed since otherwise $F$ has no minimum.)
This energy is well-studied in both the theoretical optimal transportation literature and the computational geometry literature (e.g.~\cite{DuFaberGunzburger99}) and has applications in optimal location problems (e.g., urban planning~\cite{LoschWoglomStolper54,ButtazzoSantambrogio,BouchitteJimenezMahadevan}), quantization (e.g., image compression and signal processing~\cite{Gersho79,Merigot11}) and data clustering (e.g., k-means clustering \cite{ImaiKatohInaba, DuFaberGunzburger99}). We briefly recall
how the problem of minimising $F$ can be converted into an optimal partitioning problem.

Let $\mathcal{T}_M$ be the set of partitions of $\Omega$ into $M$ sets:
\begin{equation}
\mathcal{T}_M = \left\{ \{U_i\}_{i=1}^M : U_i \subset \Omega, \; \bigcup_{i=1}^M U_i = \Omega, \; |U_i \cap U_j| = 0, i \ne j \right\}.
\end{equation}
Then $F$ can be written as
\begin{equation}
F(\nu) = \min_{\mathcal{T}_M} \left\{ \sum_{i=1}^{M} \int_{U_i} |x-x_i|^2 \, dx : |U_i|=m_i \; \forall \; i \right\}.
\end{equation}
Here $U_i = T^{-1}(\{x_i\})$ where $T$ is the optimal transportation map.
Therefore
\begin{equation}
\label{N.5}
\begin{aligned}
\min_\nu F(\nu) &= \min_{\{x_i, m_i\}_{i=1}^M} %\min_{\{\mathcal{T}_M:|V_i|=m_i\}} \; \sum_{i=1}^{M} \int_{V_i} |x-x_i|^2 \, dx
\min_{\mathcal{T}_M} \left\{ \sum_{i=1}^{M} \int_{U_i} |x-x_i|^2 \, dx : |U_i|=m_i \; \forall \; i \right\}
\\
& = \min_{\{x_i\}_{i=1}^M, \mathcal{T}_M} \; \sum_{i=1}^{M} \int_{U_i} |x-x_i|^2 \, dx.
\end{aligned}
\end{equation}
In other words, instead of minimising $F$ over points $\{ x_i \}_{i=1}^M$ and weights $\{ m_i \}_{i=1}^M$, we can minimise
\begin{equation}
\label{N.6}
G(\{x_i, U_i\}_{i=1}^M) = \sum_{i=1}^{M} \int_{U_i} |x-x_i|^2 \, dx
\end{equation}
over points $\{x_i\}_{i=1}^M$ and partitions $\{U_i\}_{i=1}^M$ of $\Omega$.
In particular
\begin{equation}
\min F(\nu) = \min G (\{x_i,U_i \})
\end{equation}
and if $\{x_i,U_i \}$ minimises $G$, then $\nu = \sum_i | U_i | \delta_{x_i}$ minimises $F$.

It is known that $G$ is minimised when the partition $\{ U_i \}_{i=1}^M$ is the Voronoi tessellation $\{ V_i \}_{i=1}^M$ generated by the points $\{ x_i \}_{i=1}^M$
and simultaneously each $x_i$ is the centre of mass of its own Voronoi cell:
\begin{align}
\label{N.7}
U_i & = V_i := \{ x \in \Omega : |x-x_i| \le |x-x_j| \; \forall \; j \ne i \},
\\
\label{N.8}
x_i & = \frac{1}{|V_i|} \int_{V_i} x \, dx,
\end{align}
for $i=1,\dots,M$.
If condition \eqref{N.7} is not satisfied it is easy to see that $G$ can be decreased by replacing the partition $\{U_i\}_{i=1}^M$
with the Voronoi tessellation $\{V_i\}_{i=1}^M$ generated by the points $\{x_i\}_{i=1}^M$.
Condition \eqref{N.8} is just the Euler-Lagrange equation $\partial G / \partial x_i = 0$ evaluated at $\{ U_i \} = \{ V_i \}$.
These special types of Voronoi tessellations $\{x_i,V_i\}$ satisfying \eqref{N.7} and \eqref{N.8} are known as \emph{Centroidal Voronoi Tessellations} (CVTs). See \cite{DuFaberGunzburger99} for a nice survey of CVTs.

The assumption that $\Omega$ is convex ensures that all the Voronoi cells are convex, since they are the intersection of $\Omega$ with half planes, and
 so the centroid of each Voronoi cell lies in $\Omega$. Without this assumption it is possible that \eqref{N.7} and \eqref{N.8} have no solution, i.e., that there exists no CVT of $M$ points of $\Omega$. This is the case, e.g., if $\Omega$ is an annulus and $M=1$. Under the assumption that $\Omega$ is convex, equations \eqref{N.7} and \eqref{N.8} always have a solution.

Note however that in general there is not a unique CVT of $M$ points -- see Figure~\ref{fig:TwoCVTs}.  In general, as $M \to \infty$, the lowest energy CVT of $M$ points tends to a hexagonal tiling. This and several extensions and  generalisations were proved in \cite{Fejes-Toth72,Gruber99,G-Fejes-Toth01,MorganBolton02}.
\begin{figure}[!h]
\centering
\includegraphics[height=3.5cm]{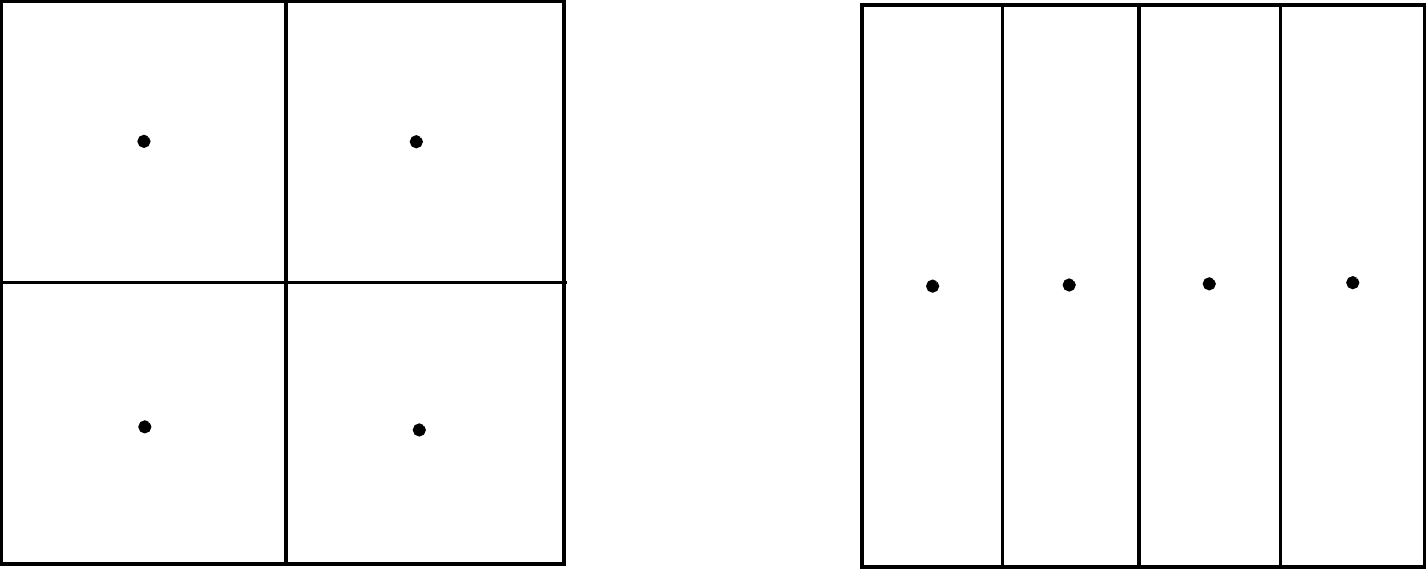}
\caption{Two partitions of the square: both are Centroidal Voronoi Tesselations of four points. The left one has lower energy $G$.}
\label{fig:TwoCVTs}
\end{figure}

CVTs can be computed using Lloyd's algorithm~\cite{Lloyd82}, which is a fixed-point method: Given an approximate set of points $\{ x_i^n \}_{i=1}^M$, compute the corresponding Voronoi diagram $\{ V_i^n \}_{i=1}^M$, and then define a new set of  points $\{ x_i^{n+1} \}_{i=1}^M$ to be the centres of mass of the Voronoi cells $\{ V_i^n \}_{i=1}^M$. Fixed points of this algorithm satisfy \eqref{N.7} and \eqref{N.8}. See \cite{DuEmelianenkoJu06, DuFaberGunzburger99} for more details and convergence theory.

For $\lambda>0$, minimisers of the original energy $F_2$ (defined in equation \eqref{N.1}) turn out to be very close to CVTs, as illustrated in Figure \ref{CVTs}. Therefore CVTs (which are easy to compute using Lloyd's algorithm) can be used to generate good approximate minimisers. Developing a convergent minimisation algorithm, however, requires a different approach, which we discuss in the following section.
\begin{figure}[!h]
\begin{center}
\includegraphics[width=0.85\textwidth]{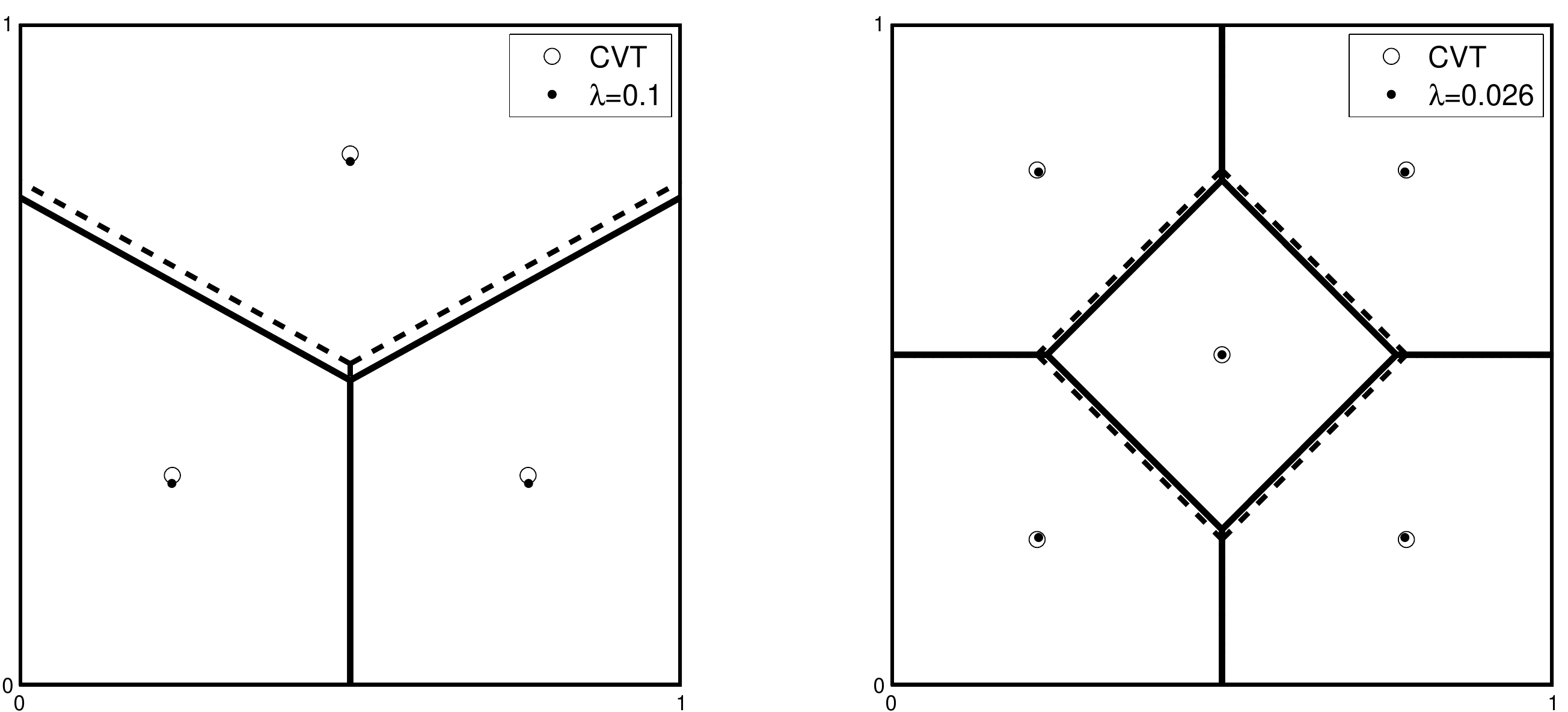}
\caption{\label{CVTs} {Minimisers of $F_2$ (defined in equation \eqref{N.1}) are close to being Centroidal Voronoi Tessellations (CVTs). Left: A CVT of three points (empty circles indicate the generators of the Voronoi cells, dashed lines indicate the boundaries of the Voronoi cells) and a minimiser $\nu$ of $F_2$ for $\lambda=0.1$ (solid circles indicate the support $\{x_i\}$ of $\nu$, solid lines indicate the boundaries of the transport regions). Right: A CVT of five points and a minimiser of $F_2$ for $\lambda=0.026$.}}
\end{center}
\end{figure}

\begin{Remark}[The case $p \in [1,\infty)$]
\label{p ne 2}
A large part of the discussion above holds for the general case $p \in [1,\infty)$. In particular, minimising $W_p(1,\nu)$ with respect to $\nu = \sum_{i=1}^M m_i \delta_{x_i}$ for $M$ fixed is equivalent to minimising
\begin{equation}
\nonumber
G_p(\{x_i, U_i\}_{i=1}^M) = \sum_{i=1}^{M} \int_{U_i} |x-x_i|^p \, dx
\end{equation}
over points $\{x_i\}_{i=1}^M$ and partitions $\{U_i\}_{i=1}^M$ of $\Omega$. As above, minimisers of $G_p$ are Voronoi diagrams
$\{x_i, V_i\}_{i=1}^M$, but this time
the generators $\{ x_i \}_{i=1}^M$ satisfy
\begin{equation}
\nonumber
\frac{\partial G_p}{\partial x_i} = 0 \iff \int_{V_i} (x_i-x) |x_i-x|^{p-2} \, dx = 0.
\end{equation}
These generalised CVTs are more difficult to compute than standard CVTs (the case $p=2$), but in principle Lloyd's algorithm could be modified
as follows: Given an approximate set of points $\{ x_i^n \}_{i=1}^M$, compute the corresponding Voronoi diagram $\{ V_i^n \}_{i=1}^M$, and then define the new set of points $\{ x_i^{n+1} \}_{i=1}^M$ to be the solutions of
\begin{equation}
\label{p centroid}
\int_{V^n_i} (x^{n+1}_i-x) |x^{n+1}_i-x|^{p-2} \, dx = 0, \quad i \in \{ 1, \ldots, M \}.
\end{equation}
The difficulty is that for general $p$ this equation is much harder to solve than for $p=2$. One option would be to use a numerical method. Another option would be to lag the nonlinear factor in \eqref{p centroid}, i.e., replace the factor $|x^{n+1}_i-x|^{p-2}$ with $|x^{n}_i-x|^{p-2}$, to obtain
\begin{equation}
\nonumber
 x_i^{n+1} = \dfrac{\displaystyle \int_{V_i^n} x |x^{n}_i-x|^{p-2} \, dx}{\displaystyle \int_{V_i^n} |x^{n}_i-x|^{p-2} \, dx}.
\end{equation}
A fixed point $\{x_i,V_i\}_{i=1}^M$ of this method is a stationary point of $G_p$, and the corresponding measure $\nu = \sum_{i=1}^M |V_i| \delta_{x_i}$
is a stationary point of $W_p(1,\nu)$.
\end{Remark}

\begin{Remark}[The case $d=1$ and $\lambda > 0$]
\label{Remark d=1}
For the 1-dimensional case, $d=1$, the functional $F_p$ defined in equation \eqref{N.1} reduces to
\begin{equation}
\nonumber
F_p(\nu) = \lambda M + W_p(1,\nu).
\end{equation}
The method given above for the case $\lambda=0$ can be applied to minimise this energy for all $\lambda > 0$, in which case the value of $M$ is not fixed \emph{a priori}; it is determined by $\lambda$. As above, stationary points of $F_p$ are Centroidal Voronoi Tesselations. In one dimension these are just
uniform partitions of the interval $\Omega$ with the points $x_i$ at the centres of the partitions. For example, if $\Omega = [0,1]$, then the stationary
points of $F_p$ are $\nu = \sum_{i=1}^M \frac 1M \delta_{x_i}$ with $x_i = \frac{1}{2M}+\frac{i-1}{M}$, for all $M \in \mathbb{N}$. These have energy
\begin{equation}
\nonumber
F_p(\nu) = \lambda M + \sum_{i=1}^M \int_{\tfrac{i-1}{M}}^{\tfrac iM} \left| x - \left( \tfrac{1}{2M}+\tfrac{i-1}{M} \right) \right|^p \, dx
= \lambda M + \frac{2^{-p}}{p+1} M^{-p}.
\end{equation}
We can find the global minimiser of $F_p$ by minimising this over $M$. This gives
\begin{equation}
\nonumber
M = \left( \frac{2^p (p+1) \lambda}{p} \right)^{-\frac{1}{p+1}}.
\end{equation}
Since $M$ is an integer, the optimal value is obtained by rounding this expression up or down, depending on the corresponding value of the energy.
We see that the optimal number of masses $M$ scales like $\lambda^{-\frac{1}{p+1}}$. In the following section we treat the case $\lambda > 0$
for $d>1$.
\end{Remark}

%
%
%----------------------------------------------------------------------------------------------------------------------------------------------
%
%

\subsection{Minimisers when $\lambda >0$: Centroidal Power Diagrams.}
\label{CPD}
Minimising the energy $F_p$ numerically is challenging, not only because it has infinitely many local minimisers, but also because it is challenging even to evaluate $F_p(\nu)$ for a given $\nu$; from the Kantorovich formulation of the $p$-Wasserstein cost $W_p(1,\nu)$, given in terms of measures on $\Omega \times \Omega$ with marginals $1$ and $\nu$, we see that evaluating $W_p(1,\nu)$ is
equivalent to solving an infinite-dimensional linear programming problem.
For the case $\lambda=0$ studied in the previous section this could be avoided. We will show that it can also be avoided for $\lambda>0$ by
rewriting the energy in new coordinates.

\begin{Definition}[$p$-power diagrams]
\label{Def}
Let $\{ x_i , w_i\}_{i=1}^M$ be a set of weighted points, $x_i \in \Omega$, $w_i \in \mathbb{R}$. To this we can associate a type of generalised Voronoi diagram $\{ P_i \}_{i=1}^M$: For $i \in \{1, \ldots, M\}$, define
\[
P_i = \{ x \in \Omega : |x-x_i|^p - w_i \le |x-x_j|^p - w_j \; \forall \; j \ne i\}.
\]
This gives a partition of $\Omega$, which we refer to as a $p$-\emph{power diagram}. We refer to the sets $P_i$ as $p$-\emph{power cells}.
\end{Definition}
If $p=2$ this is just the standard \emph{power diagram} and $P_i$ are convex polygons. For general $p$ the cells $P_i$ are not convex and their boundaries are not straight lines, unless all the weights $w_i$ are equal, in which case $\{ P_i \}_{i=1}^M$ is just a standard \emph{Voronoi diagram}. The $1$-power diagram is known in the literature as the \emph{Apollonius diagram}, \emph{hyperbolic Dirichlet tessellation}, \emph{additively weighted Voronoi diagram}, or \emph{Voronoi diagram of discs}. For general $p$ there does not seem to be a standard name and hence we call them $p$-power diagrams. These fall into the class of \emph{generalised Dirichlet tessellations} or \emph{generalised additively weighted Voronoi diagrams} studied in \cite{AshBolker86}. For a comprehensive treatment of generalised Voronoi diagrams see \cite{OkabeBootsSugiharaChiu}.

The following proposition generalises \cite[Theorem 1]{Merigot11} and \cite{AurenhammerHoffmannAronov98} from the case $p = 2$ to all $p \in [1,\infty)$:
\begin{Proposition}[Characterisation of transport regions as $p$-power diagrams.]
\label{char}
Let $\rho \in L^1(\Omega;(0,\infty))$ and let $\nu = \sum_{i=1}^M m_i \delta_{x_i}$ with $x_i \in \Omega$, $m_i \ge 0$, $\sum_i m_i = \int_\Omega \rho \, dx$.
Fix $p \in [1,\infty)$. Let $\{ U_i\}_{i=1}^M$ be the optimal transport regions for $W_p(\rho \, dx,\nu)$, i.e.,
\[
W_p(\rho \, dx,\nu) = \sum_{i=1}^M \int_{U_i} |x-x_i|^p \rho \, dx
\]
and $\int_{U_i} \rho \, dx = m_i$. Let $(\phi,\psi)$ be an optimal Kantorovich potential pair for
$W_p(\rho \, dx,\nu)$, i.e.,
\[
W_p(\rho \, dx,\nu) = \int_{\Omega} \phi (x) \rho \, dx + \sum_i m_i \psi(x_i)
\]
and
\begin{equation}
\label{ineq:prop-Kant-pair}
\phi(x) + \psi(x_i) \le |x-x_i|^p\qquad \text{for a.e. $x \in \Omega$ and all $i \in \{1, \ldots, M \}$}.
\end{equation}
Then
\begin{itemize}
\item[(i)]
$\{ U_i \}_{i=1}^M$ is the $p$-power diagram with generators $x_i$ and weights $ \psi(x_i) $:
\[
U_i = \{ x \in \Omega : |x-x_i|^p - \psi(x_i) \le |x-x_j|^p - \psi(x_j) \; \forall \; j \ne i\}.
\]
\item[(ii)]
If $\{P_i \}_{i=1}^M$ is a power diagram with generators $\{ x_i \}_{i=1}^M$ and weights $\{ w_i \}_{i=1}^M$, then
\[
W_p \left( \rho \, dx, \sum_{i} |P_i| \delta_{x_i} \right) = \sum_i \int_{P_i} |x-x_i|^p \rho \, dx.
\]
\end{itemize}
\end{Proposition}
\begin{proof}
(i) Let $P_i = P_i( \{ x_j, \psi(x_j) \}_{j=1}^M )$ be the $i$-th $p$-power cell generated by $\{ x_j, \psi(x_j) \}_{j=1}^M$.
First note that if $\{ S_i \}_{i=1}^M$ is any partition of $\Omega$, then
\begin{equation}
\label{ineq:SiPi}
\sum_{i=1}^M \int_{S_i} \left[ |x-x_i|^p - \psi(x_i) \right] \rho \, dx
\ge
\sum_{i=1}^M \int_{P_i} \left[ |x-x_i|^p - \psi(x_i) \right] \rho \, dx
\end{equation}
with equality if and only if $S_i=P_i$ for all $i$ (up to sets of Lebesgue measure zero). This follows from the definition of the $p$-power cells $P_i$. Then
\begin{eqnarray*}
W_p(\rho \, dx, \nu) & = &\sum_{i=1}^M \int_{U_i} |x-x_i|^p \rho \, dx
\\
& = &\sum_{i=1}^M \left\{ \int_{U_i} \left[ |x-x_i|^p - \psi(x_i) \right] \rho \, dx + m_i \psi(x_i) \right\}
\\
& \stackrel{\eqref{ineq:SiPi}}\ge &\sum_{i=1}^M \left\{ \int_{P_i} \left[ |x-x_i|^p - \psi(x_i) \right] \rho \, dx + m_i \psi(x_i) \right\}
\\
&\stackrel{\eqref{ineq:prop-Kant-pair}} \ge& \sum_{i=1}^M \left\{ \int_{P_i} \phi(x) \rho \, dx + m_i \psi(x_i) \right\}
\\
& = &W_p(\rho \, dx, \nu).
\end{eqnarray*}
Therefore all the inequalities above are equalities and so $U_i=P_i$ for all $i$, as required.

(ii) Now let $\{ P_i \}_{i=1}^M$ be the $p$-power diagram with generators $\{ x_i \}_{i=1}^M$ and weights $\{ w_i \}_{i=1}^M$. Let $U_i$
be the optimal partition for $W_p \left( \rho \, dx, \sum_{i} |P_i| \delta_{x_i} \right)$. Then $|U_i|=|P_i|$ and
\begin{eqnarray*}
W_p \left( \rho \, dx, \sum_{i} |P_i| \delta_{x_i} \right) & = &\sum_{i=1}^M \int_{U_i} |x-x_i|^p \rho \, dx
\\
& = &\sum_{i=1}^M \left\{ \int_{U_i} \left[ |x-x_i|^p - w_i \right] \rho \, dx + |U_i| w_i \right\}
\\
& \stackrel{\eqref{ineq:SiPi}}\ge &\sum_{i=1}^M \left\{ \int_{P_i} \left[ |x-x_i|^p - w_i \right] \rho \, dx + |U_i| w_i \right\}
\\
& = &\sum_{i=1}^M \int_{P_i} |x-x_i|^p \rho \, dx
\\
& \ge& W_p \left( \rho \, dx, \sum_{i} |P_i| \delta_{x_i} \right)
\end{eqnarray*}
since $\{ P_i \}_{i=1}^M$ is an admissible partition for $W_p \left( \rho \, dx, \sum_{i} |P_i| \delta_{x_i} \right)$. Therefore
the inequalities above are equalities, yielding the desired result.
\end{proof}

The following theorem generalises \cite[Lemma 1]{BournePeletierTheil} from the case $p = 2$ to all $p \in [1,\infty)$:
\begin{Theorem}[Global minimisers of $F_p$ are centroidal $p$-power diagrams.]
\label{GlobalMin}
Let $\nu = \sum_{i=1}^M m_i \delta_{x_i}$ be a global minimiser of $F_p$ $($defined in equation \eqref{N.1}$)$ such that $x_i \in \Omega$ for all $i$, i.e., $x_i \notin \partial \Omega$. Let $\{ P_i \}_{i=1}^M$ be the $p$-power diagram with generators $\{ x_i \}_{i=1}^M$ and weights
\[
w_i = - \lambda \frac{(d-1)}{d} m_i^{-\frac{1}{d}}, \qquad i \in \{ 1, \ldots , M \}.
\]
Then, for each $i$, $x_i$ is the $p$-centroid of $P_i$ and $m_i$ is the $d$-dimensional Lebesgue measure of $P_i$:
\begin{equation}
\label{NLE}
\int_{P_i} (x_i-x) |x_i-x|^{p-2} \, dx = 0, \qquad m_i = |P_i|.
\end{equation}
\end{Theorem}
\begin{proof}
This follows immediately from Propositions \ref{Prop1}, \ref{Prop2} and \ref{char}.
\end{proof}

Observe that \eqref{NLE} is a pair of nonlinear equations for global minimisers $\{x_i, m_i \}_{i=1}^M$. The rest of this section is devoted to the numerical solution of these equations.
For implementation purposes we will focus on the case $p=2$ since $2$-power diagrams (which are just power diagrams) are easy to compute. We also limit our attention to two space dimensions, $d=2$, and to convex domains $\Omega$.

By Proposition \ref{char}, minimising $F_2$ over atomic measures $\nu$ is equivalent to minimising the following energy $E$ over sets of points and weights $\{ x_i,w_i \}$:
\begin{equation}
\label{N.9}
E(\{x_i, w_i\}) := F_2 \left( \sum_{i} |P_i| \delta_{x_i} \right) = \sum_i \left\{ \lambda \sqrt{|P_i|} + \int_{P_i} |x-x_i|^2 \, dx \right\},
\end{equation}
where $\{P_i\}$ is the power diagram generated by $\{x_i,w_i\}$. The energy $E$ can be evaluated to machine precision, whereas evaluating the original energy $F_2$ directly involves solving an infinite-dimensional linear programming problem.

\paragraph{A generalised Lloyd algorithm.}
By Theorem \ref{GlobalMin}, global minimisers of $E$ (which correspond to global minimisers of $F_2$) are centroidal power diagrams and are fixed points of the following algorithm:
Given an approximate set of points and weights $\{ x_i^n,w_i^n \}_{i=1}^M$, first compute the corresponding power diagram $\{ P_i^n \}_{i=1}^M$.
(If any of the power cells $P_i^n$ are empty, then delete the corresponding points and weights and update $M$.)
Then define a new set of points and weights $\{ x_i^{n+1},w_i^{n+1} \}_{i=1}^M$ by
\begin{equation}
\label{algorithm}
\begin{aligned}
x_i^{n+1} & = \frac{1}{|P^n_i|} \int_{P_i^n} x \, dx, \\
w_i^{n+1} & = - \frac 12 \lambda |P^n_i|^{-\frac 12}.
\end{aligned}
\end{equation}
This is a generalisation of Lloyd's algorithm, which we described in Section \ref{lambda=0}.
The assumption that $\Omega$ is convex ensures that $x_i^{n+1}$ lies in $\Omega$.
Note that fixed points of this algorithm are not necessarily global minima. Below we discuss the implementation issue of how to find global minima as opposed to only local minima. It is easy to generalise this algorithm to $p \in [1,\infty)$, although harder to implement on a computer since $p$-power diagrams are more difficult to compute.

Figure \ref{Minimisers} (in Section \ref{Intro}) shows minimisers of $F_2$ for decreasing values of $\lambda$ computed using this method.
Observe that as $\lambda \to 0$ the power diagrams tend towards a hexagonal tiling of $\Omega$, and the generators $x_i$ tend to a triangular lattice. In Section \ref{P} we will see that this limiting behaviour can be proved rigorously.

We will extend this algorithm to a more general class of optimal location problems, study convergence, and implement it in three dimensions in a forthcoming paper \cite{BourneRoper}.

\paragraph{Implementation: Searching for global minimisers.}
Global minimisers of $E$ as $\lambda\rightarrow 0$ asymptotically approach a regular hexagonal tiling of $\Omega$ (see Section \ref{P}). Therefore if we assume that boundary effects are small then the interior of $\Omega$ should be approximately tiled with hexagons of an appropriate size when $\lambda$ is small. For a single regular hexagon with diameter $D$, with generator placed at the centre, the area is $A_D=3\sqrt{3}D^2/8$, the 2-Wasserstein cost is $C_D=5\sqrt{3}D^4/128$, and the total energy is
\begin{equation}
\label{eqn:epera}
e_D=\lambda3^{3/4}2^{-3/2}D+3^{1/2}2^{-7}5D^4.
\end{equation}
The total energy per unit area is $e_D/A_D=\lambda3^{-3/4}2^{3/2}D^{-1}+5D^2/48$. Minimising this gives an optimal diameter $\tilde{D}=3^{1/12}2^{3/2}5^{-1/3}\lambda^{1/3}$ and correspondingly $A_{\tilde{D}}^{-1}\lambda^{2/3}=5^{2/3}3^{-5/3}\approx 0.4685737$. Therefore the total energy of $|\Omega| A_{\tilde{D}}^{-1}$ hexagons each with energy $e_{\tilde{D}}$ is
\begin{equation}
\label{eqn:totale}
E_{\tilde{D}}= e_{\tilde{D}} |\Omega| A_{\tilde{D}}^{-1} = \frac{1}{2}5^{1/3}3^{1/6}\lambda^{2/3}|\Omega|.
\end{equation}
As $\lambda\rightarrow 0$ we expect the number of cells of a global minimising state to be approximately $M_g:=|\Omega|A_{\tilde{D}}^{-1}$. The energy landscape is extremely flat and the number of stable local minima increases as $\lambda\rightarrow 0$. Our algorithm is energy decreasing \cite{BourneRoper} but there is no guarantee that the state to which it converges is a global minimum. We employ a crude genetic algorithm in an attempt to find a global minimum.

For $\lambda>0$, we estimate the number of cells of a global minimum to be $M_g$. For each integer $M$ in the interval
$I =[M_g- C\lambda^{-2/3},M_g+ C\lambda^{-2/3}]$,
 where $C$ is chosen based on experiments, we distribute $M$ points $x_i$ uniformly in $\Omega$ with initial weights $w_i=0$. For each integer value of $M \in I$ we generate $N_r$ such random states and use \eqref{algorithm} to find a local minimum. After a prescribed number of iterations the collection of results are sorted in order of increasing energy and states that have both the same energy (to a certain tolerance) and the same number of cells are factored out. Then the remaining states are improved using \eqref{algorithm} and the process repeated until successive iterates differ in their energy by less than a prescribed tolerance. In this way we pursue only the best candidates. This method is crude, but gives reasonable results for moderate $\lambda$.

Figure \ref{fig:globalalg} shows the result of applying our algorithm with $\lambda=0.005$ in the unit square. Several such experiments were performed and gave the same five lowest energy states, with differences in the sixth state. In all experiments for $\lambda=0.005$ the same lowest energy state was obtained, this state is shown in Figure \ref{Minimisers} (top right). As can be seen from Figure \ref{fig:globalalg}, the energies of the local minimisers are very close and so the number of random initial states that must be tested increases as $\lambda$ gets small.
\begin{figure}
\begin{center}
\includegraphics{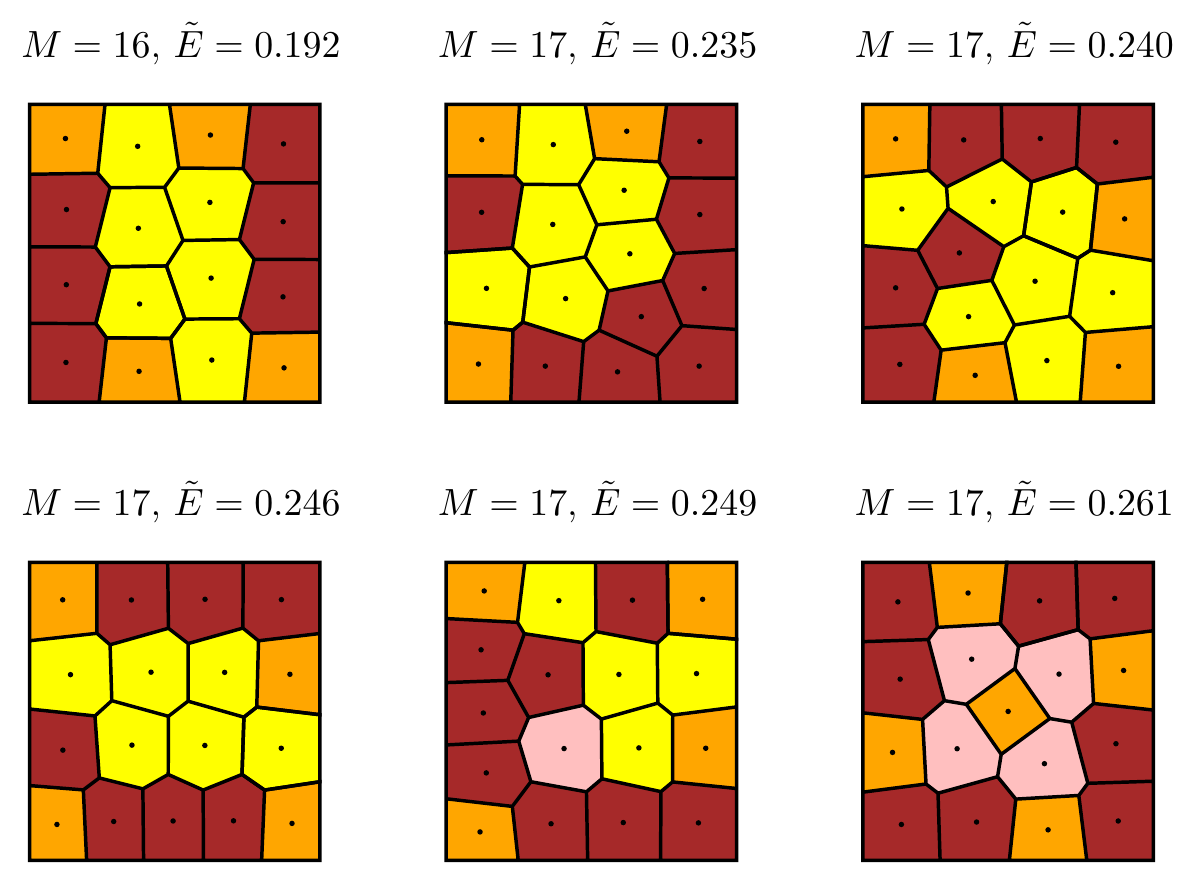}
\end{center}
\caption{\label{fig:globalalg} The six lowest energy states with $\lambda=0.005$, using $N_r=250$. The cells are coloured according to the number of sides. The energy $\tilde{E}$ is a rescaled version of the energy $E$: $\tilde{E}=(E-E_{\tilde{D}})/e_{\tilde{D}}$, where $E_{\tilde{D}}$ is defined in \eqref{eqn:totale} and $e_{\tilde{D}}$ is defined in \eqref{eqn:epera}.}
\end{figure}

For small $\lambda$, it is difficult to find a global minimum and experiments show that the lowest energy states calculated by our genetic algorithm are characterised by different `grains' -- regions of regular hexagonal tiling -- that intersect at grain boundaries.
See Figure \ref{fig:grains}.
Distortions of the regular hexagonal pattern from both the grain boundaries and the boundary of $\Omega$ decay over a length of two of three cells.

\begin{figure}[h]
\begin{center}
\includegraphics{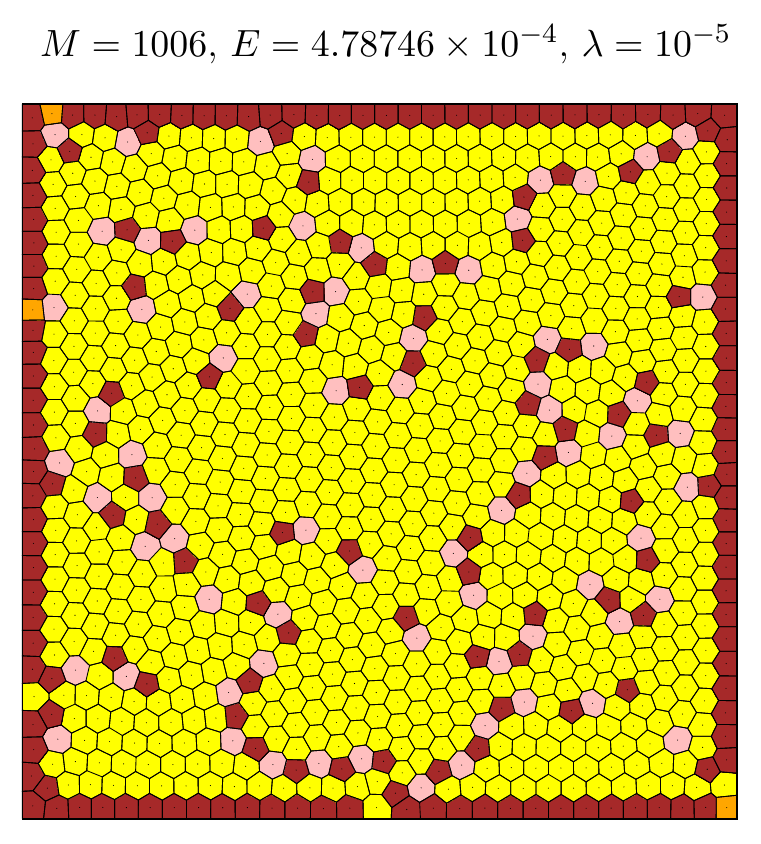}
\end{center}
\caption{\label{fig:grains} `Grains' of hexagonal tiling. This figure was produced from a completely random (uniformly distributed) initial state using 1000 iterations of the generalised Lloyd algorithm given in equation \eqref{algorithm}. Colouring indicates number of neighbours or sides: yellow polygons have six sides, pink seven, red five, and orange four.}
\end{figure}

\begin{Remark}[Alternative algorithm]
For the case $p=2$, an alternative strategy for minimising $F_2(\nu)$ is the following: Approximate the Lebesgue measure $1$ in $W_2(1,\nu)$ with an atomic measure $\mu$ representing a Gaussian quadrature rule of degree $2$ on $\Omega$; degree $2$ is specially tailored to the quadratic transport cost $|x-y|^2$. Then the value of $W_2(1,\nu)$ can be well approximated by $W_2(\mu,\nu)$, which is a \emph{finite-dimensional} linear programming problem. Therefore, while $F_2(\nu)$ cannot be evaluated exactly, it can be approximated well by $\lambda \sum \sqrt{m_i} + W_2(\mu,\nu)$, and a standard nonlinear optimisation algorithm could be used to minimise this. Even with this sensible choice of discretisation, however, the method given above is much faster and more accurate.
\end{Remark}

%
%
%
%
%------------------------------------------------------------------------------------------------------
%
%
%
%

\section{Exact Characterisation of Minimisers in 2D}
\label{P}
Let $d=2$ and let $\nu_\lambda$ be a minimiser of $F_2$ (defined in equation \eqref{N.1}). The numerical results in the previous section suggest that, as $\lambda \to 0$, the support $\{ x_i \}_{i=1}^{M_\lambda}$ of $\nu_\lambda$ tends to a triangular lattice, and the associated transport regions $V_i$ tend to a regular hexagonal tiling of $\Omega$. A precise statement of this is proved in the companion paper
\cite{BournePeletierTheil}. Here we just give a rough statement.
Let $\Omega \subset \mathbb{R}^2$ have area 1.
We rescale $\Omega$ to obtain a domain $\Omega_\lambda$ that blows up as $\lambda \to 0$:
\[
\Omega_\lambda := \left( \frac{2 c_6}{\lambda} \right)^\frac{1}{3} \Omega
\]
where $c_6 := \frac{5 \sqrt{3}}{54}$ is the cost of transporting the Lebesgue measure restricted to a unit area regular hexagon onto a dirac mass located as its centre. Under this rescaling the energy $F_2$ becomes, up to a factor $\lambda^{4/3}(2 c_6)^{-4/3}$, the following:
\[
\tilde{F_2}(\nu) = 2 c_6 \sum_{i=1}^M \sqrt{m_i} + W_2(1,\nu)
\]
where $\nu = \sum_{i=1}^M m_i \delta_{x_i}$ is a measure on the rescaled domain $\Omega_\lambda$.
Then \cite[Theorems 1--4]{BournePeletierTheil} can be stated roughly as follows:
\begin{itemize}
\item Let $\Omega$ be a polygonal domain with at most six sides. For $\lambda >0$, the energy $\tilde{F_2}$ is bounded from below by the energy of
a measure supported on a triangular lattice (to be more precise, the energy of a measure supported at the centres of $|\Omega_\lambda|$ unit-area regular hexagons).
\item This lower bound can be achieved in the limit $\lambda \to 0$. (This statement holds for any Lipschitz domain $\Omega$.) It can also be achieved for $\lambda>0$ if $\Omega$ is a periodic domain of the right aspect ratio (meaning that it exactly fits an integer number of unit-area regular hexagons).
\item If the energy of $\nu = \sum_i m_i \delta_{x_i}$ is close to the lower bound, then $\{ x_i \}$ is close to being a triangular lattice.
\end{itemize}

%
%
%
%
%-------------------------------------------------------------------------------------------------------
%
%
%
%

\section{Conjecture About Minimisers in 3D}
\label{3D}
In experiments on diblock copolymers where one phase has a very low volume fraction, it is observed that the minority phase
forms small spheres embedded in a sea of the majority phase. These spheres are centred on a BCC (body-centred cubic) lattice (see e.g.~\cite{BatesFredrickson}). We conjecture that minimisers of $F_2$ in three dimensions are indeed BCC lattices. There is tantalizing evidence for this. We demonstrated numerically (see Fig.~\ref{CVTs}) that in two dimensions minimisers of $F_2$ are close to being centroidal Voronoi tessellations (CVTs), meaning that the points $\{ x_i \}$ in the support of the optimal measure $\nu$ are close to being the generators of a CVT. We expect the same to be true in three dimensions. It is not known what the optimal CVTs are in three dimensions, where optimal means that they minimise $G$ (defined in \eqref{N.6}), but there is strong numerical evidence to suggest that they are generated by the BCC lattice~\cite{DuWang05}. It has also been proved the lowest energy \emph{lattice} CVT (i.e., the lowest energy CVT generated by a lattice) in three dimensions is generated by the BCC lattice~\cite{BarnesSloane83}.

Therefore we conjecture the following:
 Let $d=3$ and $\nu_\lambda$ be a minimiser of $F_2$. As $\lambda \to 0$, the support $\{ x_i \}_{i=1}^{M_\lambda}$ of $\nu_\lambda$ tends to a BCC lattice.
 Proving this seems out of reach, but it will be studied numerically in a forthcoming paper.

%
%-----------------------------------------------------------------------------------------------------------------------
%

\bigskip

\noindent{\textbf{Acknowledgments.}} Much of the work of D.~P. Bourne was carried out while he held a postdoc position at the
Technische Universiteit Eindhoven, supported by the grant `Singular-limit Analysis of Metapatterns', NWO grant 613.000.810. The results presented in Section \ref{P} were obtained in collaboration with Florian Theil.

%
%------------------------------------------------------------------------------------------------------------
%

\bibliography{refsBournePeletierRoper}
\bibliographystyle{plain}

\end{document}